\documentclass[reqno,11pt]{amsart}
\usepackage{amssymb}
\usepackage{mathrsfs}
\usepackage{amsmath,amssymb}
\usepackage{paralist}
\usepackage{graphics} 
\usepackage{epsfig} 
\usepackage[colorlinks=true]{hyperref}
\hypersetup{urlcolor=red, citecolor=blue}
\usepackage[top=1in,bottom=1in,left=1in,right=1in]{geometry}
\usepackage{cite}
\newtheorem{thm}{Theorem}[section]
\newtheorem{cor}[thm]{Corollary}
\newtheorem{lem}[thm]{Lemma}

\theoremstyle{definition}
\newtheorem{defn}{Definition}[section]
\newtheorem{rem}{Remark}[section]
\numberwithin{equation}{section}
\newcommand{\be}{\begin{equation}}
\newcommand{\ee}{\end{equation}}
\begin{document}
\title[CNS system with slow $p$-Laplacian diffusion]
{Global weak solutions for the three-dimensional chemotaxis-Navier-Stokes system with slow $p$-Laplacian diffusion}%
\author[Tao]{Weirun Tao}%
\address{Department of Mathematics, Southeast University, Nanjing 210096, P. R. China}
\email{taoweiruncn@163.com}
\author[Li]{Yuxiang Li}%
\address{Department of Mathematics, Southeast University, Nanjing 210096, P. R. China}
\email{lieyx@seu.edu.cn}
\thanks{Supported in part by National Natural Science Foundation of China (No. 11171063, No. 11671079, No. 11701290), National Natural Science Foundation of China under Grant (No. 11601127), and National Natural Science Foundation of Jiangsu Provience (No. BK20170896).}


\begin{abstract}
 This paper investigates an incompressible chemotaxis-Navier-Stokes system with slow $p$-Laplacian diffusion
 \begin{eqnarray}\nonumber
  \left\{\begin{array}{lll}
     \medskip
     n_t+u\cdot\nabla n=\nabla\cdot(|\nabla n|^{p-2}\nabla n)-\nabla\cdot(n\chi(c)\nabla c),&{} x\in\Omega,\ t>0,\\
     \medskip
     c_t+u\cdot\nabla c=\Delta c-nf(c),&{} x\in\Omega,\ t>0,\\
     \medskip
     u_t+\kappa(u\cdot\nabla) u=\Delta u+\nabla P+n\nabla\Phi,&{} x\in\Omega,\ t>0,\\
     \medskip
     \nabla\cdot u=0,&{} x\in\Omega,\ t>0
  \end{array}\right.
\end{eqnarray}
under homogeneous boundary conditions of Neumann type for $n$ and $c$, and of Dirichlet type for $u$ in a bounded convex domain $\Omega\subset \mathbb{R}^3$ with smooth boundary. Here, $\Phi\in W^{1,\infty}(\Omega)$, $0<\chi\in C^2([0,\infty))$ and $0\leq f\in C^1([0,\infty))$ with $f(0)=0$. It is proved that if $p>\frac{32}{15}$ and under appropriate structural assumptions on $f$ and $\chi$, for all sufficiently smooth initial data $(n_0,c_0,u_0)$ the model possesses at least one global weak solution.
\end{abstract}
\maketitle
\section{Introduction}
In this paper, we consider the following chemotaxis-Navier-Stokes system with $p$-Laplacian diffusion
 \begin{eqnarray}\label{CNS}
  \left\{\begin{array}{lll}
     \medskip
     n_t+u\cdot\nabla n=\nabla\cdot\left(|\nabla n|^{p-2}\nabla n\right)-\nabla\cdot(n\chi(c)\nabla c),&{} x\in\Omega,\ t>0,\\
     \medskip
     c_t+u\cdot\nabla c=\Delta c-nf(c),&{}x\in\Omega,\ t>0,\\
     \medskip
     u_t+\kappa(u\cdot\nabla)u=\Delta u+\nabla P+n\nabla\Phi ,&{}x\in\Omega,\ t>0,\\
     \medskip
     \nabla\cdot u=0, &{}x\in\Omega,\ t>0
  \end{array}\right.
\end{eqnarray}
in a smooth bounded domain $\Omega\subset\mathbb{R}^3$, where the scalar functions $n =n(x, t)$ and $c=c(x, t)$ denote bacterial density and the concentration of oxygen, respectively. The vector $u =(u_1, u_2, u_3)$ is the fluid velocity field and the associated pressure is denoted by $P=P(x, t)$. The function $\chi$ represents the chemotactic sensitivity, $f$ is the oxygen consumption rate by the bacteria and $\kappa\in \mathbb{R}$ measures the strength of nonlinear fluid convection. The given function $\Phi$ stands for the gravitational potential produced by the action of physical forces on the cell. If $p>2$, the nonlinear diffusion $|\nabla n|^{p-2}\nabla n$ is called the slow $p$-Laplacian diffusion, whereas $1 < p < 2$, it is called the fast $p$-Laplacian diffusion.
\vskip 3mm
The Keller-Segel model was first presented in \cite{KS1970} to describe the chemotaxis of cellular slime molds. Let $n$ denote the cell density and $c$ describe the concentration of the chemical signal secreted by cells. The mathematical model reads as
 \begin{eqnarray}
  \left\{\begin{array}{lll}
     \medskip
     n_t=\Delta n-\nabla\cdot(n\nabla c),&{} x\in\Omega,\ t>0,\\
     \medskip
     c_t=\Delta c-c+n,&{}x\in\Omega,\ t>0,
  \end{array}\right.
\end{eqnarray}
which is called Keller-Segel system. It is known that whether the classical solutions of the system exist globally depends on the size of initial data (cf.\cite{Cao-DCDS-2015,Winkler-JDE-2010,Nagai&Senba&Yoshida-FE-1997,Horstmann&Wang-JAM-2001}). A large number of variants of the classical form have been investigated, including the system with the logistic terms (see \cite{Tao&Winkler-SIAMJMA-2011,Zhang&Li-ZAMP-2015-B,LiYan&Lankeit-N-2016}, for instance), two-species chemotaxis system (see \cite{Lou&Tao&Winkler-SIAMJMA-2014,Zhang&Li-JMAA-2014,LiYan&Li-NA-2014,LiYan-JMAA-2015}, for instance), attraction-repulsion chemotaxis system (see\cite{Tao&Wang-M3AS-2013,LiYan&Li-NARWA-2016}, for instance) and so on. We refer to \cite{Bellomo&Bellouquid&Tao&Winkler-M3AS-2015,Hillen-JMB-2009,Horstmann-JDMV-2003,Horstmann-JDMV-2004} for the further reading.

The chemotaxis-Navier-Stokes system was first introduced in \cite{Tuval2005}. Aerobic bacteria such as \emph{Bacillus subtilis} often live in thin fluid layers near solid-air-water contact line, in which the biology of chemotaxis, metabolism, and
cell-cell signaling is intimately connected to the physics of buoyancy,
diffusion, and mixing\cite{Tuval2005}. Both bacteria and oxygen diffuse through the fluid, and they are also transported by the fluid (cf. \cite{DOMBROWSKI-PRL-2004} and \cite{Lorz-M3AS-2010}). Taking into account all these processes, in \cite{Tuval2005} the authors proposed the model
\begin{eqnarray}\label{model-Tuval}
  \left\{\begin{array}{lll}
     \medskip
     n_t+u\cdot\nabla n=\Delta n-\nabla\cdot(n\chi(c)\nabla c),&{} x\in\Omega,\ t>0,\\
     \medskip
     c_t+u\cdot\nabla c=\Delta c-nf(c),&{} x\in\Omega,\ t>0,\\
     \medskip
     u_t+\kappa(u\cdot\nabla) u=\Delta u+\nabla P+n\nabla\Phi,&{} x\in\Omega,\ t>0,\\
     \medskip
     \nabla\cdot u=0,&{} x\in\Omega,\ t>0
  \end{array}\right.
\end{eqnarray}
in a domain $\Omega\subset \mathbb{R}^d$, where the vector $u=(u_1(x,t), u_2(x,t), \cdots, u_d(x,t))$ is the fluid velocity field and the associated pressure is represented by $P=P(x,t)$.

The chemotaxis fluid system has been studied in the last few years. In \cite{Lorz-M3AS-2010}, local-in-time weak solutions were constructed for a boundary-value problem of (\ref{model-Tuval}) in the three-dimensional setting. In \cite{Duan&Lorz&Markowich-CPDE-2010}, global classical solutions near constant states were constructed with $\Omega=\mathbb{R}^3$. For the chemotaxis-Navier-Stokes system in two space dimensions, the authors in \cite{Liu&Lorz-AIHP-2011} obtained global existence for large data. For the case of bounded domain $\Omega\subset\mathbb{R}^d$, Winkler \cite{Winkler-CPDE-2012} proved that the initial-boundary value problem of (\ref{model-Tuval}) possesses a unique global classical solution for $d=2$ and possesses at least one global weak solution for $d=3$ under the assumption that $\kappa=0$. In \cite{Winkler-ARMA-2014} the same author showed that in bounded convex domains $\Omega\subset \mathbb{R}^2$, the global classical solutions obtained in \cite{Winkler-CPDE-2012} stabilize to the spatially uniform equilibrium $(\bar{n}_0, 0, 0)$ with $\bar{n}_0:=\frac1{|\Omega|}\int_\Omega n_0(x)dx$ as $t\rightarrow\infty$. Recently, Zhang and Li \cite{Zhang&Li-DCDS-2015} proved that such solution converges to the equilibrium $(\bar{n}_0, 0, 0)$ exponentially in time. By deriving a new type of entropy-energy estimate, Jiang et al. \cite{Jiang&Wu&Zheng-2014} generalized the result of \cite{Winkler-ARMA-2014} to general bounded domains. (If both $\chi$ and $f$ are supposed to be nonnegative and nondecreasing, it was shown by Chae, Kang and Lee \cite{Chae-DCDS-2013} that the Cauchy problem of (\ref{model-Tuval}) admits a global classical solution under the assumption that $d=2$ and $\sup_c|\chi(c) -\mu f(c)|$ be sufficiently small for some $\mu >0$. It was showed in \cite{LiYan&Li-JDE-2016} that the $2$-dimensional Cauchy problem of (\ref{model-Tuval}) admits global classical bounded solutions for regular initial data. For more results of the well-posedness of the Cauchy problem to (\ref{model-Tuval}) in the whole space we refer the reader to \cite{Chae-CPDE-2014,Duan&Lorz&Markowich-CPDE-2010,Liu&Lorz-AIHP-2011,ZhangQian-NARWA-2014,Zhang&Zheng-SIAM-2014}.)

The diffusion of bacteria sometimes depend nonlinearly on their densities \cite{Hillen-JMB-2009,Vazquez-2007,Tao&Winkler-AIHP-2013,Tao&Winkler-DCDS-2012}. Introducing this into the model (\ref{model-Tuval}) leads to the chemotaxis-Navier-Stokes system with nonlinear diffusion\cite{Di&Lorz&Markowich-DCDS-2010}
\begin{eqnarray}\label{p-laplacian-ks}
  \left\{\begin{array}{lll}
     \medskip
     n_t+u\cdot\nabla n=\nabla\cdot\left(D(n)\nabla n\right)-\nabla\cdot(n\chi(c)\nabla c),&{} x\in\Omega,\ t>0,\\
     \medskip
     c_t+u\cdot\nabla c=\Delta c-nf(c),&{}x\in\Omega,\ t>0,\\
     \medskip
     u_t+\kappa(u\cdot\nabla)u=\Delta u+\nabla P+n\nabla\Phi ,&{}x\in\Omega,\ t>0,\\
     \medskip
     \nabla\cdot u=0, &{}x\in\Omega,\ t>0.
  \end{array}\right.
\end{eqnarray}
Under the assumption $D(n)=n^{m-1}$, Di Francesco et al. \cite{Di&Lorz&Markowich-DCDS-2010} proved that when $m\in(\frac32, 2]$ the chemotaxis-Stokes system admits a global-in-time solution for general initial data in the bounded domain $\Omega\subset \mathbb{R}^2$, while the same result holds in three-dimensional setting under the constraint $m \in(\frac{7+\sqrt{217}}{12}, 2]$. Intuitively, the nonlinear diffusion $\Delta u^m$ for $m>1$ can prevent the occurrence of blow up. Based on this intuition, Winkler \cite{Winkler-CVPDE-2015} revealed the condition $m>\frac76$ to be sufficient to guarantee the boundedness of global weak solutions to the chemotaxis-stokes system for all reasonably regular initial data in three-dimensional
bounded convex domains (see also \cite{Liu&Lorz-AIHP-2011}). This partially extended a precedent result which asserted
global solvability within the larger range $m>\frac87$ , but only in a class of weak solutions locally bounded in $\bar{\Omega}\times [0,\infty)$ (cf.\cite{Tao&Winkler-AIHP-2013}). In \cite{Zhang&Li-JDE-2015}, Zhang and Li studied the system (\ref{p-laplacian-ks}) under the assumption $D(n)=mn^{m-1}$ and they proved that the model possesses at least one global weak solution under the condition that $m\geq\frac23$. Recently, Winkler \cite{Winkler-ARXIV-201704} considered (\ref{p-laplacian-ks}) under the assumption that $\chi(s)\equiv1$, $f(s)\equiv s$ for $s\geq0$ and that $\kappa=0$. It was shown that the chemotaxis-Stokes system admits global bounded weak solutions to an associated initial-boundary value problem under the assumption that $m>\frac98$. Moreover, the obtained solutions are shown to approach the spatially homogeneous steady state ($\frac1{|\Omega|}\int_\Omega n_0,0,0$) in the large time limit. For smaller values of $m > 1$, up to now existence results are limited to classes of possibly unbounded solutions (cf.\cite{Duan&Xiang-IMRN-2014}).
\vskip 3mm
Other forms of diffusion operator are also considered in Keller-Segel model, and $p$-Laplacian diffusion is one of them. Cong et al. \cite{Cong&Liu-KRM-2016} studied the following $p$-Laplacian Keller-Segel model in $d\geq3$:
 \begin{eqnarray}\label{ljgmodel}
  \left\{\begin{array}{lll}
     \medskip
     n_t=\nabla\cdot\left(|\nabla n|^{p-2}\nabla n\right)-\nabla\cdot(n\chi(c)\nabla c),&{} x\in\mathbb{R}^d,\ t>0,\\
     \medskip
     \Delta c=n,&{}x\in\mathbb{R}^d,\ t>0,\\
     \medskip
     n(x,0)=n_0(x), &{}x\in\mathbb{R}^d,
  \end{array}\right.
\end{eqnarray}
and they proved the existence of a uniform in time $L^\infty$ bounded weak solution for system (\ref{ljgmodel}) with the supercritical
diffusion exponent $1<p<\frac{3d}{d+1}$ in the multi-dimensional space $\mathbb{R}^d$
under the condition that the $L^{\frac{d(3-p)}p}$ norm of initial data is smaller than a
universal constant. They also proved the local existence of weak solutions and a
blow-up criterion for general $L^1\cap L^\infty$ initial data. By the way, Jian-Guo et al. \cite{Hui&Liu-KRM-2016} investigated the system (\ref{ljgmodel}) with the nonlocal diffusion term instead by $-(-\Delta)^{\frac\alpha2}n$ ($1<\alpha<2$) in dimension $d> 2$.

In contrast to the chemotaxis-(Navier-)Stokes system with porous-medium-type diffusion or the classical Keller-Segel model with $p$-Laplacian diffusion, very few results of global solvability are available for the full chemotaxis-Navier-Stokes system with $p$-Laplacian diffusion. This inspires us to study (\ref{CNS}).
\vskip 3mm
\textbf{Main results.} 
In order to formulate our result, we specify the precise mathematical setting: we shall subsequently consider (\ref{CNS}) along with initial conditions
\be
n(x,0)=n_0(x),\ \ c(x,0)=c_0(x)\ \ {\rm and}\ \ u(x,0)=u_0(x),\ \ \ \ x\in\Omega\label{eq-ic}
\ee
and under the boundary conditions
\be
|\nabla n|^{p-2}\frac{\partial n}{\partial \nu}=\frac{\partial c}{\partial \nu}=0\ \ {\rm and}\ \ u=0\ \ \ \ {\rm{on}}\ \partial\Omega\label{eq-bc}
\ee
in a bounded convex domain $\Omega\subset\mathbb{R}^3$ with smooth boundary, where $\nu$ is the exterior unit normal vector on $\partial\Omega$ and we assume that
\begin{eqnarray}\label{eq-initialdata-n-c-u}
  \left\{\begin{array}{lcl}
     \medskip
     n_{0}\in L^2(\Omega)\ \ \ \ {\rm is\ nonnegative\ with}\ n_0\not\equiv0,\ \ \ \ {\rm that}\\
     \medskip
     c_0\in L^\infty(\Omega)\ {\rm is\ nonnegative\ and\ such\ that} \sqrt{c_0}\in W^{1,2}(\Omega),\ \ \ \ {\rm and\ that}\\
     \medskip
     u_0\in L^2_\sigma(\Omega)
  \end{array}\right.
\end{eqnarray}
with $L_{\sigma}^2(\Omega):=\left\{\varphi\in (L^2(\Omega))^3|\nabla\cdot\varphi=0\right\}$ denotes the Hilbert space of all solenoidal vector in $L^2(\Omega)$.

With respect to the parameter function in (1.2), we shall suppose throughout the paper that
\begin{eqnarray}\label{eq-chi-f-phi}
  \left\{\begin{array}{lcl}
     \medskip
     \chi\in C^2([0,\infty)),\ \ \chi>0\ \ {\rm{in}}\ [0,\infty),\ \ \ \ {\rm that}\\
     \medskip
     f\in C^1([0,\infty)),\ \ f(0)=0,\ \ f>0\ \ {\rm{in}}\ [0,\infty),\ \ \ \ {\rm and\ that}\\
     \medskip
     \Phi\in W^{1,\infty}(\Omega),
  \end{array}\right.
\end{eqnarray}
as well as
\be
(\frac f\chi)'>0,\ \ (\frac f\chi)''\leq0,\ \ {\rm{and}}\ \ (\chi\cdot f)'\geq0\ \ \ \ {\rm{on}}\ [0,\infty).\label{eq-chi-f}
\ee

Our main result reads as follows.
\begin{thm}\label{main result}
Let $\Omega\subset\mathbb{R}^3$ be a bounded convex domain with smooth boundary, $\kappa\in\mathbb{R}$ and $p>\frac{32}{15}$. Suppose that the assumptions (\ref{eq-initialdata-n-c-u})-(\ref{eq-chi-f}) hold. Then there exists at least one global weak solution (in the sense of Definition \ref{defnaaaa} below) of (\ref{CNS}), (\ref{eq-ic}) and (\ref{eq-bc}) such that
\begin{equation*}
n\in L_{loc}^{p}([0, \infty); W^{1,p}(\Omega))\quad\mbox{and}\quad c^{\frac{1}{4}}\in L_{loc}^{4}([0, \infty); W^{1,4}(\Omega)).
\end{equation*}
\end{thm}
\begin{rem}\label{remone} Compared with the result of \cite{Zhang&Li-JDE-2015}, one may wonder why the condition of $p$ in this paper is not less than 2, but larger than $\frac{32}{15}$. There are two reasons. One is a technical reason, and the other is more essential: the main difficulties appear when we want to show that the solutions $(n_\varepsilon, c_\varepsilon, u_\varepsilon)$ of regularized system satisfies
\be
|\nabla n_\varepsilon|^{p-2}\nabla n_\varepsilon\rightharpoonup |\nabla n|^{p-2}\nabla n\ \ \ \ \mbox{in}\ L_{loc}^{q}(\bar{\Omega}\times[0, \infty))
\ee
for some $q>1$. This requires a higher regularity of the solutions $(n_\varepsilon, c_\varepsilon, u_\varepsilon)$ than that of \cite{Zhang&Li-JDE-2015}.
Monotonic method (cf. \cite[Section 2.1]{J.L.Lions-DGVP-1969}, see also \cite{Bendahmane2009}) is used to handle this problem, but it's crucial to derive the a priori estimates. In fact, we need to show that
\begin{equation}
\int^t_0\int_\Omega|\nabla n_\varepsilon|^p\leq C(t+1)\ \ \ \ {\rm{for\ all}}\ \ t>0,\label{eqkksss}
\end{equation}
but using the same method used in \cite{Winkler-AIHP-2016} and \cite{Zhang&Li-JDE-2015}, we can not get the desired result, which can only obtain
\begin{equation}
\int^t_0\int_\Omega|\nabla n_\varepsilon|^{p-\frac34}\leq C(t+1)\ \ \ \ {\rm{for\ all}}\ \ t>0.\label{eqkss}
\end{equation}
We overcome this difficulty by means of a bootstrap argument, but we need a stronger hypotheses: we need $n_0\in L^2(\Omega)$ rather than $n_0\in L\log L(\Omega)$.
And also the condition $p>\frac{32}{15}$ is required for the sake of deriving a crucial priori estimates. Actually, we shall prove that (see Section 6)
\be
|\nabla n_\varepsilon|^{p-2}\nabla n_\varepsilon\rightharpoonup |\nabla n|^{p-2}\nabla n\ \ \ \ \mbox{in}\ L_{loc}^{p'}(\bar{\Omega}\times[0, \infty)).
\ee
\end{rem}

The rest of this paper is organized as follows. In Section 2, we introduce a family of regularized problems and give some preliminary properties. Based on an energy-type inequality, a priori estimates are given in Section 3. Section 4 is devoted to showing the global existence of the regularized problems. In Section 5, we further establish useful uniform estimates. Finally, we give the proof of the main result in Section 6.
\vskip 3mm
{\it Notations.} Throughout the paper, for any vectors $v\in \mathbb{R}^3$ and $w\in \mathbb{R}^3$, we denote by $v\otimes w$ the matrix $A_{3\times3}$ with $a_{ij}=v_iw_j$ for $i,j\in\{1,2,3\}$. We set $L\log L(\Omega)$ is the standard Orlicz space and $L_{\sigma}^2(\Omega):=\left\{\varphi\in (L^2(\Omega))^3|\nabla\cdot\varphi=0\right\}$ denotes the Hilbert space of all solenoidal vector in $L^2(\Omega)$. As usual $\mathcal {P}$ denotes the Helmholtz projection in $L^2(\Omega)$. We write $W_{0, \sigma}^{1,2}(\Omega):=W_{0}^{1,2}(\Omega)\cap L_{\sigma}^2(\Omega)$ and $C_{0,\sigma}^{\infty}(\Omega):=C_{0}^{\infty}(\Omega)\cap L_{\sigma}^2(\Omega)$. We represent $A$ as the realization of Stokes operator $-\mathcal {P}\Delta$ in $L_{\sigma}^2(\Omega)$ with domain $D(A):=W^{2,2}(\Omega)\cap W_{0, \sigma}^{1,2}(\Omega)$. Also $n(\cdot,t)$, $c(\cdot,t)$ and $u(\cdot,t)$ will be denoted sometimes by $n(t)$, $c(t)$ and $u(t)$.
\section{Regularized problem}
Our intention is to construct a global weak solution as the limit of smooth solutions of appropriately regularized problem. According to the idea from \cite{Winkler-AIHP-2016} (see also \cite{Tao&Winkler-AIHP-2013,Zhang&Li-JDE-2015}), let us first consider the approximate problem
\begin{eqnarray}\label{model-approximate}
  \left\{\begin{array}{lcl}
     \medskip
       \partial_t{n_\varepsilon}+{u_\varepsilon}\cdot\nabla {n_\varepsilon}=\nabla\cdot\left((|\nabla {n_\varepsilon}|^2+\varepsilon)^\frac{p-2}2\nabla {n_\varepsilon}\right)-\nabla\cdot({n_\varepsilon}{F_\varepsilon}'({n_\varepsilon})\chi({c_\varepsilon})\nabla {c_\varepsilon}),&& x\in\Omega,\ t>0,\\
     \medskip
     \partial_t {c_\varepsilon}+{u_\varepsilon}\cdot\nabla {c_\varepsilon}=\Delta {c_\varepsilon}-{F_\varepsilon}({n_\varepsilon})f({c_\varepsilon}),&& x\in\Omega,\ t>0,\\
     \medskip
     \partial_t{u_\varepsilon}+({Y_\varepsilon}{u_\varepsilon}\cdot\nabla) {u_\varepsilon}=\Delta {u_\varepsilon}+\nabla P_\varepsilon+{n_\varepsilon}\nabla\Phi,&& x\in\Omega,\ t>0,\\
     \medskip
     \nabla\cdot {u_\varepsilon}=0,&& x\in\Omega,\ t>0,\\
     \medskip
     \frac{\partial {n_\varepsilon}}{\partial \nu}=\frac{\partial {c_\varepsilon}}{\partial \nu}=0,\ {u_\varepsilon}=0,  && x\in\partial\Omega,\ t>0,\\
     \medskip
     {n_\varepsilon}(x,0)=n_{0\varepsilon},\ {c_\varepsilon}(x,0)=c_{0\varepsilon},\ {u_\varepsilon}(x,0)=u_{0\varepsilon},\   && x\in\Omega
  \end{array}\right.
\end{eqnarray}
for $\varepsilon\in(0,1)$, where the families of approximate initial data $n_{0\varepsilon}\geq 0$, $c_{0\varepsilon}\geq 0$ and $u_{0\varepsilon}$ have the properties that
\begin{eqnarray}\label{approximate-initialdata-n}
  \left\{\begin{array}{lcl}
     \medskip
     n_{0\varepsilon}\in C^\infty_0(\Omega),\ \ \int_\Omega n_{0\varepsilon}=\int_\Omega n_0\ \ \ \ {\rm{for\ all}}\ \varepsilon\in(0,1)\ \ \ \ {\rm and}\\
     \medskip
     n_{0\varepsilon}\rightarrow n_0\ \ {\rm{in}}\ L^2(\Omega)\ \ \ \ as\ \ \varepsilon\searrow0,
  \end{array}\right.
\end{eqnarray}
that
\begin{eqnarray}\label{approximate-initialdata-c}
  \left\{\begin{array}{lcl}
     \medskip
     \sqrt{c_{0\varepsilon}}\in C^\infty_0(\Omega),\ \ \|c_{0\varepsilon}\|_{L^\infty(\Omega)}\leq\|c_{0}\|_{L^\infty(\Omega)}\ \ \ \ {\rm{for\ all}}\ \varepsilon\in(0,1)\ \ \ \ {\rm and}\\
     \medskip
     \sqrt{c_{0\varepsilon}}\rightarrow \sqrt{c_0}\ \ a.e.\ {\rm{in}}\ \Omega\ \ and\ \ W^{1,2}(\Omega)\ \ \ \ as\ \ \varepsilon\searrow0,
  \end{array}\right.
\end{eqnarray}
and that
\begin{eqnarray}\label{approximate-initialdata-u}
  \left\{\begin{array}{lcl}
     \medskip
     u_{0\varepsilon}\in C^\infty_{0,\sigma}(\Omega),\ \ {\rm{with}}\ \ \|u_{0\varepsilon}\|_{L^2(\Omega)}=\|u_{0}\|_{L^2(\Omega)}\ \ \ \ {\rm{for\ all}}\ \varepsilon\in(0,1)\ \ \ \ {\rm and}\\
     \medskip
     u_{0\varepsilon}\rightarrow u_0\ \ \ {\rm{in}}\ L^2(\Omega)\ \ \ \ as\ \ \varepsilon\searrow0.
  \end{array}\right.
\end{eqnarray}
The approximate function ${F_\varepsilon}$ in (\ref{model-approximate}) \cite{Winkler-AIHP-2016} was chosen as
\begin{eqnarray*}
{F_\varepsilon}(s):=\frac1\varepsilon\ln(1+\varepsilon s),\ \ \ \ {\rm{for\ all}}\ s\geq0,
\end{eqnarray*}
and the standard Yosida approximate ${Y_\varepsilon}$ \cite{Sohr-2001-book} was defined by
\begin{eqnarray*}
{Y_\varepsilon} v:=(1+\varepsilon A)^{-1}v,\ \ \ \ {\rm{for\ all}}\ v\in L^2_\sigma(\Omega).
\end{eqnarray*}
Let us furthermore note that the choice of ${F_\varepsilon}$ ensures that
\be
0\leq{F_\varepsilon}'(s)=\frac1{1+\varepsilon s}\leq1,\ \ {\rm{and}}\ \ 0\leq{F_\varepsilon}(s)\leq s\ \ \ \ {\rm{for\ all}}\ s\geq0,\label{eq-F-eq1}
\ee
\be
s{F_\varepsilon}'(s)=\frac s{1+\varepsilon s}\leq\frac1\varepsilon,\ \ \ \ {\rm{for\ all}}\ s\geq0,\label{eq-F-eq2}
\ee
and that
\be
{F_\varepsilon}'(s)\nearrow1\ \ {\rm{and}}\ \ {F_\varepsilon}(s)\nearrow s\ \ \ \ as\ \ \ \ \varepsilon\searrow0\ \ \ \ {\rm{for\ all}}\ s\geq0.\label{eq-F-eq3}
\ee
All the above approximate problems admit for local-in-time smooth solutions:
\begin{lem}\label{lem2.1}
Let $p\geq2$, then for each $\varepsilon\in(0,1)$, there exist $T_{max,\varepsilon}\in(0,\infty]$ and functions ${n_\varepsilon}>0$, ${c_\varepsilon}>0$ in $\bar{\Omega}\times(0,T_{max,\varepsilon})$, and ${u_\varepsilon}$ fulfilling
\begin{eqnarray}
  \left.\begin{array}{lcl}
     \medskip
       {n_\varepsilon}\in C^0(\bar{\Omega}\times[0,T_{max,\varepsilon}))\cap C^{2,1}(\bar{\Omega}\times(0,T_{max,\varepsilon})),\\
     \medskip
     {c_\varepsilon}\in C^0(\bar{\Omega}\times[0,T_{max,\varepsilon}))\cap C^{2,1}(\bar{\Omega}\times(0,T_{max,\varepsilon})),\ \ \ \ and\\
     \medskip
     {u_\varepsilon}\in C^0(\bar{\Omega}\times[0,T_{max,\varepsilon});\mathbb{R}^3)\cap C^{2,1}(\bar{\Omega}\times(0,T_{max,\varepsilon});\mathbb{R}^3)
  \end{array}\right.
\end{eqnarray}
such that $({n_\varepsilon},{c_\varepsilon},{u_\varepsilon})$ is a classical solution of (\ref{model-approximate}) in $\Omega\times(0,T_{max,\varepsilon})$ with some $P_\varepsilon\in C^{1,0}(\Omega\times(0,T_{max,\varepsilon}))$. Moreover, if $T_{max,\varepsilon}<\infty$, then
\be
\|{n_\varepsilon}(\cdot,t)\|_{L^\infty(\Omega)}+\|{c_\varepsilon}(\cdot,t)\|_{W^{1,q}(\Omega)}+\|A^\alpha{u_\varepsilon}(\cdot,t)\|_{L^2(\Omega)}\rightarrow\infty,\ \ t\nearrow T_{max,\varepsilon}\nonumber
\ee
for all $q>3$ and $\alpha>\frac34$.
\end{lem}
The proof of this lemma is based on well-established methods involving the Schauder fixed point theorem, the standard regularity theories of parabolic equations and the Stokes system (see \cite{Tao&Winkler-AIHP-2013,Winkler-CPDE-2012,Winkler-AIHP-2016} for instance).

The following estimates of ${n_\varepsilon}$ and ${c_\varepsilon}$ are basic but important in the proof of our result.
\begin{lem}
For each $\varepsilon\in(0,1)$, the solution of (\ref{model-approximate}) satisfies
\be
\int_\Omega{n_\varepsilon}(\cdot,t)=\int_\Omega n_0\ \ \ \ for\ all\ t\in(0,T_{max,\varepsilon})\label{eq-n-mass-conservation}
\ee
and
\be
\|{c_\varepsilon}(\cdot,t)\|_{L^\infty(\Omega)}\leq\| c_0\|_{L^\infty(\Omega)}:=s_0\ \ \ \ 
for\ all\ t\in(0,T_{max,\varepsilon}).\label{eq-c-maxmum-principle-s0}
\ee
\end{lem}
\begin{proof}
Integrating the first equation in (\ref{model-approximate}) and using (\ref{approximate-initialdata-n}), we obtain (\ref{eq-n-mass-conservation}).
And an application of the maximum principle to the second equation in (\ref{model-approximate}) gives (\ref{eq-c-maxmum-principle-s0}).
\end{proof}
\section{An energy-type inequality}
This section is devoted to establish an energy-type inequality which will play a key role in the derivation of further estimates in Section \ref{sectionfour} where we will show that the solution of the approximate problem (\ref{model-approximate}) is actually global in time.

The following inequality is decisive in our proof, which is derived from the first two equations in (\ref{model-approximate}). The main idea of the proof is similar
to the strategy introduced in \cite[Section 3]{Winkler-AIHP-2016}(see also \cite{Winkler-CPDE-2012,Zhang&Li-JDE-2015}). 
\begin{lem}\label{lem3.1}
Assume that $p\geq2$. Let (\ref{eq-chi-f-phi}) and (\ref{eq-chi-f}) hold. There exists $K\geq1$ such that for any $\varepsilon\in(0,1)$, the solution of (\ref{model-approximate}) satisfies
\begin{eqnarray}
&&\frac d{dt}\left(\int_\Omega{n_\varepsilon}\ln {n_\varepsilon}+\frac12\int_\Omega|\nabla\Psi({c_\varepsilon})|^2\right)\nonumber\\
&&+\frac1K\left\{\int_\Omega(|\nabla{n_\varepsilon}|^2+\varepsilon)^\frac{p-2}2\frac{|\nabla{n_\varepsilon}|^2}{{n_\varepsilon}}+\int_\Omega\frac{|D^2{c_\varepsilon}|^2}{{c_\varepsilon}}+\int_\Omega\frac{|\nabla{c_\varepsilon}|^4}{c^3_\varepsilon}\right\}\nonumber\\
&\leq& K\int_\Omega|\nabla{u_\varepsilon}|^2,\label{eq-lem3.1}
\end{eqnarray}
where $\Psi(s):=\int^s_1\frac{d\sigma}{\sqrt{g(\sigma)}}$ with $g(s):=\frac{f(s)}{\chi(s)}$.
\end{lem}
\begin{proof}
The proof is based on the first two equations in (\ref{model-approximate}) and  integration by parts,
we refer readers to \cite[Lemmas 3.1-3.4]{Winkler-AIHP-2016} for details.
\end{proof}

Using the third equation of (\ref{model-approximate}), we can absorb the term on the right hand of (\ref{eq-lem3.1}) appropriately and yield
the following energy-type inequality which contains all the components ${n_\varepsilon}$, ${c_\varepsilon}$ and ${u_\varepsilon}$.
\begin{lem}\label{lem3.2}
Assume that $p\geq2$. Let (\ref{eq-chi-f-phi}), (\ref{eq-chi-f}) hold, and $\Psi$ and $K$ be given in Lemma \ref{lem3.1}. Then for any $\varepsilon\in(0,1)$ there exist $K_0>0$ large enough such that
\begin{eqnarray}
&&\frac d{dt}\left(\int_\Omega{n_\varepsilon}\ln {n_\varepsilon}+\frac12\int_\Omega|\nabla\Psi({c_\varepsilon})|^2+K\int_\Omega|{u_\varepsilon}|^2\right)\nonumber\\
&&+\frac1{K_0}\left\{\int_\Omega|\nabla{n_\varepsilon}^{\frac{p-1}p}|^p+\int_\Omega\frac{|D^2{c_\varepsilon}|^2}{{c_\varepsilon}}+\int_\Omega\frac{|\nabla{c_\varepsilon}|^4}{c^3_\varepsilon}+\int_\Omega|\nabla {u_\varepsilon}|^2\right\}\nonumber\\
&\leq&K_0\ \ \ \  for\ all\ t\in(0,T_{max,\varepsilon}).\label{eq-lem3.2}
\end{eqnarray}
\end{lem}
\begin{proof}
Testing the third equation of (\ref{model-approximate}) by ${u_\varepsilon}$ and integrating by parts over $\Omega$, we have
$$\frac12\frac d{dt}\int_\Omega|{u_\varepsilon}|^2+\int_\Omega|\nabla {u_\varepsilon}|^2=-\int_\Omega({y_\varepsilon}{u_\varepsilon}\cdot\nabla){u_\varepsilon}\cdot{u_\varepsilon}+\int_\Omega{n_\varepsilon}\nabla\Phi\cdot{u_\varepsilon}
\ \ \ \ {\rm for\ all}\ t\in(0,T_{max,\varepsilon}).$$
Since $\nabla\cdot{u_\varepsilon}=0$ implies $\nabla\cdot{Y_\varepsilon}{u_\varepsilon}=0$, we thereby obtain (see also \cite[Lemma 3.5]{Winkler-AIHP-2016})
\begin{eqnarray}
\frac12\frac d{dt}\int_\Omega|{u_\varepsilon}|^2+\int_\Omega|\nabla {u_\varepsilon}|^2=\int_\Omega{n_\varepsilon}\nabla\Phi\cdot{u_\varepsilon}
\ \ \ \ {\rm for\ all}\ t\in(0,T_{max,\varepsilon}).\nonumber
\end{eqnarray}
Substituting this into (\ref{eq-lem3.1}), we have
\begin{eqnarray}
&&\frac{d}{dt}\left(\int_\Omega{n_\varepsilon}\ln {n_\varepsilon}+\frac12\int_\Omega|\nabla\Psi({c_\varepsilon})|^2+K\int_\Omega|{u_\varepsilon}|^2\right)\nonumber\\
&&+\frac1K\left\{\int_\Omega(|\nabla{n_\varepsilon}|^2+\varepsilon)^\frac{p-2}2\frac{|\nabla{n_\varepsilon}|^2}{{n_\varepsilon}}+\int_\Omega\frac{|D^2{c_\varepsilon}|^2}{{c_\varepsilon}}+\int_\Omega\frac{|\nabla{c_\varepsilon}|^4}{c^3_\varepsilon}\right\}+K\int_\Omega|\nabla{u_\varepsilon}|^2\nonumber\\
&\leq& 2K\int_\Omega {n_\varepsilon}\nabla\Phi\cdot{u_\varepsilon}\ \ \ \ {\rm for\ all}\ t\in(0,T_{max,\varepsilon}).\label{eq-proof-lem3.2-eq1}
\end{eqnarray}
Since $p\geq2$, we have for each $\varepsilon\in(0,1)$
\be
\int_\Omega(|\nabla{n_\varepsilon}|^2+\varepsilon)^\frac{p-2}2\frac{|\nabla{n_\varepsilon}|^2}{{n_\varepsilon}}\geq\int_\Omega\frac{|\nabla{n_\varepsilon}|^p}{{n_\varepsilon}}
=(\frac{p}{p-1})^p\int_\Omega|\nabla{n_\varepsilon}^{\frac{p-1}p}|^p\ \ \ \ {\rm for\ all}\ t\in(0,T_{max,\varepsilon}).\label{eq-proof-lem3.2-eq2}
\ee
Using H${\rm{\ddot{o}}}$lder's inequality, (\ref{eq-chi-f-phi}), the Sobolev embedding $W^{1,2}(\Omega)\hookrightarrow L^6(\Omega)$ and the
standard Poincar${\rm{\acute{e}}}$ inequality, one can find a constant $C_1>0$ such that
\begin{eqnarray}
2K\int_\Omega {n_\varepsilon}\nabla\Phi\cdot{u_\varepsilon}&\leq&2K\|\nabla\Phi\|_{\infty}\|{n_\varepsilon}\|_{\frac65}\|{u_\varepsilon}\|_6\nonumber\\
&\leq&\frac K2\|\nabla {u_\varepsilon}\|_2^2+C_1\|{n_\varepsilon}\|_{\frac65}^2\ \ \ \ {\rm for\ all}\ t\in(0,T_{max,\varepsilon}).\label{eq-proof-lem3.2-eq3}
\end{eqnarray}
Let $\theta:=\frac{p-1}{4(2p-3)}\in(0,1)$, then $\theta$ satisfies $\frac{5(p-1)}{6p}=\theta(\frac1p-\frac13)+(1-\theta)\frac{p-1}p$. An application of the Gagliardo-Nirenberg inequality shows that
\begin{eqnarray}
\|{n_\varepsilon}\|_{\frac65}^2&=&\|{n_\varepsilon}^{\frac{p-1}p}\|_{\frac{6p}{5(p-1)}}^{\frac{2p}{p-1}}\nonumber\\
&\leq&C_2\|\nabla n_\varepsilon^{\frac{p-1}p}\|_p^{\frac{2p}{p-1}\theta}\|n_\varepsilon^{\frac{p-1}p}\|_{\frac p{p-1}}^{\frac{2p}{p-1}(1-\theta)}+C_2\|n_\varepsilon^{\frac{p-1}p}\|^\frac{2p}{p-1}_{\frac p{p-1}}\nonumber\\
&=&C_2\|\nabla n_\varepsilon^{\frac{p-1}p}\|_p^{\frac{2p}{p-1}\theta}\|n_0\|_1^{2{(1-\theta)}}+C_2\|n_0\|_1^2\ \ \ \ {\rm for\ all}\ t\in(0,T_{max,\varepsilon})\label{eq-proof-lem3.2-eq5}
\end{eqnarray}
with some $C_2>0$. Since ${\frac{2p}{p-1}\theta}<p$ due to our assumption $p\geq2>\frac74$, we use Young's inequality together with (\ref{eq-proof-lem3.2-eq5}) to estimate
\begin{eqnarray}
\|{n_\varepsilon}\|_{\frac65}^2\leq\delta \|\nabla{n_\varepsilon}^{\frac{p-1}p}\|_p^p+ C_3C_\delta\ \ \ \ {\rm for\ all}\ t\in(0,T_{max,\varepsilon})\label{eq-proof-lem3.2-eq6}
\end{eqnarray}
with some $C_3>0$ and some $\delta>0$ to be fixed later. Combining (\ref{eq-proof-lem3.2-eq1})-(\ref{eq-proof-lem3.2-eq6}) with $\delta=\frac1{2KC_1}(\frac{p}{p-1})^p$, we arrive at (\ref{eq-lem3.2}).
\end{proof}

We can thereby establish the following consequences of Lemma \ref{lem3.2}.
\begin{lem} \label{lem3.3}
Assume that $p\geq2$. Let (\ref{eq-chi-f-phi}), (\ref{eq-chi-f}) hold, and $\Psi$ and $K$ be given in Lemma \ref{lem3.1}. Then there exists $C\geq 0$ such that for any
$\varepsilon\in(0,1)$ we have
\begin{eqnarray}
\int_\Omega{n_\varepsilon}\ln {n_\varepsilon}+\frac12\int_\Omega|\nabla\psi({c_\varepsilon})|^2+K\int_\Omega|{u_\varepsilon}|^2&\leq& C\ \ \ \  for\ all\ t\in(0,T_{max,\varepsilon})\label{eq-lem3.3-1}
\end{eqnarray}
and
\begin{eqnarray}
\int^T_0\int_\Omega|\nabla{n_\varepsilon}^{\frac{p-1}p}|^p+\int^T_0\int_\Omega\frac{|D^2{c_\varepsilon}|^2}{{c_\varepsilon}}
+\int^T_0\int_\Omega\frac{|\nabla{c_\varepsilon}|^4}{c^3_\varepsilon}+\int^T_0\int_\Omega|\nabla{u_\varepsilon}|^2&\leq& C(T+1)\label{eq-lem3.3-2}
\end{eqnarray}
for all $t\in(0,T_{max,\varepsilon}).$
\end{lem}
\begin{proof}
Set
\begin{eqnarray*}
{y_\varepsilon}(t):=\int_\Omega\left\{{n_\varepsilon}\ln {n_\varepsilon}+\frac12\int_\Omega|\nabla\Psi({c_\varepsilon})|^2+K\int_\Omega|{u_\varepsilon}|^2\right\}(\cdot,t)\ \ \ \ {\rm for\ all}\ t\in(0,T_{max,\varepsilon})
\end{eqnarray*}
and
\begin{eqnarray*}
h_\varepsilon(t):=\int_\Omega\left\{|\nabla{n_\varepsilon}^{\frac{p-1}p}|^p+\int_\Omega\frac{|D^2{c_\varepsilon}|^2}{{c_\varepsilon}}+\int_\Omega\frac{|\nabla{c_\varepsilon}|^4}{c^3_\varepsilon}+\int_\Omega|\nabla{u_\varepsilon}|^2\right\}(\cdot,t)
\end{eqnarray*}
for all $t\in(0,T_{max,\varepsilon})$. Then (\ref{lem3.2}) implies that
\be
{y_\varepsilon}'(t)+\frac1{K_0}h_\varepsilon(t)\leq K_0\ \ \ \ {\rm for\ all}\ t\in(0,T_{max,\varepsilon}).\label{eq-proof-lem3.3-eq1}
\ee
In order to introduce dissipative term in (\ref{eq-proof-lem3.3-eq1}), we show that ${y_\varepsilon}(t)$ is dominated by $h_\varepsilon(t)$.
We first use the elementary inequality $z\ln z\leq2 z^{\frac65}$ for all $z\geq0$.
And invoking the Young's inequality together with (\ref{eq-proof-lem3.2-eq6}) (pick $\delta=1$), we infer that
\be
\int_\Omega{n_\varepsilon}\ln {n_\varepsilon}\leq 2\int_\Omega{n_\varepsilon}^\frac65=2\|{n_\varepsilon}\|_{\frac65}^{\frac65}\leq\|{n_\varepsilon}\|_{\frac65}^2+2^{\frac{5}2}
\leq\|\nabla{n_\varepsilon}^{\frac{p-1}p}\|_p^p+ C_7\label{eq-proof-lem3.3-eq2}
\ee
with some $C_7>0$ for all $t\in(0,T_{max,\varepsilon})$. From the standard Poincar${\rm{\acute{e}}}$ inequality, there exists $C_8>0$ such that
\be
K\int_\Omega|{u_\varepsilon}|^2\leq C_8\int_\Omega|\nabla{u_\varepsilon}|^2\ \ \ \ {\rm for\ all}\ t\in(0,T_{max,\varepsilon}).\label{eq-proof-lem3.3-eq3}
\ee
Recalling the definitions of $\Psi$ and $g$ in Lemma \ref{lem3.1}, using (\ref{eq-chi-f-phi}), we have
$$
g(s):=\frac{f(s)}{\chi(s)}\in C^1([0,\infty))\ \ \ {\rm{and}}\ \ \ g(0)=0.
$$
Hence there exist two constants $C_g^->0$ and $C_g^+>0$ such that $C_g^-s\leq g(s)\leq C_g^+s$, which yields

\begin{eqnarray}
\frac12\int_\Omega|\nabla\Psi({c_\varepsilon})|^2&=&\frac12\int_\Omega\frac{|\nabla{c_\varepsilon}|^2}{g({c_\varepsilon})}\nonumber\\
&\leq&\int_\Omega\frac{|\nabla{c_\varepsilon}|^4}{{c_\varepsilon}^3}+\frac1{16}\int_\Omega\frac{{c_\varepsilon}^3}{g^2({c_\varepsilon})}\nonumber\\
&\leq&\int_\Omega\frac{|\nabla{c_\varepsilon}|^4}{{c_\varepsilon}^3}+\frac1{16(C^-_g)^2}\int_\Omega{c_\varepsilon}\nonumber\\
&\leq&\int_\Omega\frac{|\nabla{c_\varepsilon}|^4}{{c_\varepsilon}^3}+\frac{s_0|\Omega|}{16(C^-_g)^2}\ \ \ \ {\rm for\ all}\ t\in(0,T_{max,\varepsilon}).
\end{eqnarray}
In conjunction with (\ref{eq-proof-lem3.3-eq2}) and (\ref{eq-proof-lem3.3-eq3}), we obtain
\be
{y_\varepsilon}(t)\leq C_9 h_\varepsilon(t)+C_9\ \ \ \ {\rm for\ all}\ t\in(0,T_{max,\varepsilon})\nonumber
\ee
with $C_9:=\max\{1,C_8,C_7+\frac{s_0|\Omega|}{16(C^-_g)^2}\}$, which together with (\ref{eq-proof-lem3.3-eq1}) imply that ${y_\varepsilon}$ satisfies the ODI
\begin{eqnarray}
{y_\varepsilon}'(t)+\frac1{2K_0}h_\varepsilon(t)+\frac1{2C_9K_0}{y_\varepsilon}(t)\leq K_0+\frac{1}{2K_0}:=C_{10}\ \ \ \ {\rm for\ all}\ t\in(0,T_{max,\varepsilon}).\label{eq-proof-lem3.3-eq4}
\end{eqnarray}
This firstly shows that
\begin{eqnarray}
y_\varepsilon(t)\leq \max\{\sup_{\varepsilon\in(0,1)}y_\varepsilon(0),2C_9K_0C_{10}\}\ \ \ \ {\rm for\ all}\ t\in(0,T_{max,\varepsilon}).\nonumber
\end{eqnarray}
In view of (\ref{approximate-initialdata-n})-(\ref{approximate-initialdata-u}) and \cite[Lemma 3.7]{Winkler-AIHP-2016}, we obtain (\ref{eq-lem3.3-1}).
Secondly, another integration of (\ref{eq-proof-lem3.3-eq4}) proves (\ref{eq-lem3.3-2}).
\end{proof}
\section{Global existence for the regularized problem (\ref{model-approximate})}\label{sectionfour}
Now we are in the position to show that the solution of  the approximate problem (\ref{model-approximate}) is actually global in time. The idea of the
proof is based on the argument in \cite[Lemma 3.9]{Winkler-AIHP-2016} (see also\cite[Section 4]{Zhang&Li-JDE-2015}) for the linear case $p=2$. Throughout this section,
all the constants below possibly depending on $\varepsilon$.
\begin{lem} \label{lem4.1}
Assume that $p\geq2$. For each $\varepsilon\in(0,1)$, the solutions of {\rm{(\ref{model-approximate})}} are global in time.
\end{lem}
\begin{proof} Assume that $T_{max,\varepsilon}<\infty$ for some $\varepsilon\in(0,1)$. As an particular consequence of
Lemma \ref{lem3.3} and (\ref{eq-c-maxmum-principle-s0}) we can find $C_1>0$ and $C_2>0$ such that
\be
\int^{T_{max,\varepsilon}}_0\int_\Omega|\nabla{c_\varepsilon}(\cdot,t)|^4\leq C_1\ \ \ {\rm{and}}\ \ \ \int_\Omega|{u_\varepsilon}(\cdot,t)|^2\leq C_2
\ \ \ \ {\rm for\ all}\ t\in(0,T_{max,\varepsilon}).\label{eq-p-lem4.1-eq1}
\ee

We first multiply the first equation in (\ref{model-approximate}) by $\beta{n_\varepsilon}^{\beta-1}$ with $\beta\in(1,8(1-\frac1p))$ and using integration by parts, we have
\begin{eqnarray}
&&\frac d{dt}\int_\Omega{n_\varepsilon}^\beta +\beta(\beta-1)\int_\Omega(|\nabla{n_\varepsilon}|^2+\varepsilon)^{\frac{p-2}2}{n_\varepsilon}^{\beta-2}|\nabla{n_\varepsilon}|^2\nonumber\\
&=&\beta(\beta-1)\int_\Omega {n_\varepsilon}{F_\varepsilon}'({n_\varepsilon})\chi({c_\varepsilon})\nabla {c_\varepsilon}{n_\varepsilon}^{\beta-2}\nabla{n_\varepsilon}\nonumber
\end{eqnarray}
for all $t\in(0,T_{max,\varepsilon})$. A combination of (\ref{eq-F-eq2}) and Young's inequality shows that
\begin{eqnarray}
&&\frac d{dt}\int_\Omega{n_\varepsilon}^\beta +\beta(\beta-1)\int_\Omega(|\nabla{n_\varepsilon}|^2+\varepsilon)^{\frac{p-2}2}{n_\varepsilon}^{\beta-2}|\nabla{n_\varepsilon}|^2\nonumber\\
&\leq& \frac{\beta(\beta-1)}\varepsilon \max_{s\in[0,s_0]}\chi(s)\int_\Omega n_\varepsilon^{\beta-2}|\nabla{c_\varepsilon}||\nabla{n_\varepsilon}|\nonumber\\
&\leq&\int_\Omega{n_\varepsilon}^{\beta-2}\left(\beta(\beta-1)|\nabla{n_\varepsilon}|^p+C_3|\nabla{c_\varepsilon}|^{\frac p{p-1}}\right)\nonumber
\end{eqnarray}
for all $t\in(0,T_{max,\varepsilon})$ with some $C_3>0$. Obviously, $\int_\Omega{n_\varepsilon}^{\beta-2}|\nabla{n_\varepsilon}|^p\leq\int_\Omega(|\nabla{n_\varepsilon}|^2+\varepsilon)^{\frac{p-2}2}{n_\varepsilon}^{\beta-2}|\nabla{n_\varepsilon}|^2$, which
implies that
\begin{eqnarray}
\frac d{dt}\int_\Omega{n_\varepsilon}^\beta +\beta(\beta-1)\int_\Omega{n_\varepsilon}^{\beta-2}|\nabla{n_\varepsilon}|^p
\leq \int_\Omega{n_\varepsilon}^{\beta-2}\left(\beta(\beta-1)|\nabla{n_\varepsilon}|^p+C_3|\nabla{c_\varepsilon}|^{\frac p{p-1}}\right)\nonumber
\end{eqnarray}
for all $t\in(0,T_{max,\varepsilon})$. Observing that $p>2$ implies $\frac p{p-1}<4$, we obtain there exist $C_4,C_5>0$ such that

\begin{eqnarray}
\frac d{dt}\int_\Omega{n_\varepsilon}^\beta&\leq&C_3\int_\Omega{n_\varepsilon}^{\beta-2}|\nabla{c_\varepsilon}|^{\frac p{p-1}}\nonumber\\
&\leq&C_3\int_\Omega\left(|\nabla{c_\varepsilon}|^4+C_4{n_\varepsilon}^{(\beta-2)\frac{4(p-1)}{3p-4}}\right)\nonumber\\
&\leq&C_3\int_\Omega|\nabla{c_\varepsilon}|^4+C_5\int_\Omega({n_\varepsilon}^\beta+|\Omega|)\ \ \ \ {\rm for\ all}\ t\in(0,T_{max,\varepsilon}),\label{eq-p-lem4.1-eq2}
\end{eqnarray}

since $\beta\leq 8(1-\frac1p)$ ensures that ${(\beta-2)\frac{4(p-1)}{3p-4}}\leq\beta$.
Integrating (\ref{eq-p-lem4.1-eq2}) and using (\ref{eq-p-lem4.1-eq1}) yields $C_6>0$ such that
\begin{eqnarray}
\int_\Omega{n_\varepsilon}^\beta\leq C_6\ \ \ \ {\rm for\ all}\ t\in(0,T_{max,\varepsilon}),\label{eq-p-lem4.1-eq3}
\end{eqnarray}
where $\beta\in(1,8(1-\frac1p))$.

We next use (\ref{eq-p-lem4.1-eq3}) to estimate $\|{c_\varepsilon}\|_{W^{1,\infty}}(\Omega)$. Obviously, $p>2$ implies $8(1-\frac1p)>4$, thus as a particular consequence of (\ref{eq-p-lem4.1-eq3}), we have
\begin{eqnarray}
\int_\Omega{n_\varepsilon}^4\leq C_6\ \ \ \ {\rm for\ all}\ t\in(0,T_{max,\varepsilon}),\label{eq-jfo-zzh-n-4-norm}
\end{eqnarray}
which is enough to obtain
\begin{eqnarray}
\|{u_\varepsilon}\|_{L^\infty(\Omega)}\leq C_7\ \ \ \ {\rm for\ all}\ t\in(0,T_{max,\varepsilon})\label{eq-jfo-zzh-u-infty-norm}
\end{eqnarray}
and
\begin{eqnarray}
\|\nabla{c_\varepsilon}\|_{L^4(\Omega)}\leq C_8\ \ \ \ {\rm for\ all}\ t\in(0,T_{max,\varepsilon})\label{eq-jfo-nabla-zzh-c-4-norm}
\end{eqnarray} with some $C_7>0$ and $C_8>0$. Since the proof of (\ref{eq-jfo-zzh-u-infty-norm}) and (\ref{eq-jfo-nabla-zzh-c-4-norm})
is exactly the same as in \cite[Lemma 3.9]{Winkler-AIHP-2016}, we refer the reader to it for more details. Notice that the second equation of (\ref{model-approximate}) is equivalent to
$$\partial_t {c_\varepsilon}+{u_\varepsilon}\cdot\nabla {c_\varepsilon}=(\Delta-1) {c_\varepsilon}+c_\varepsilon-{F_\varepsilon}({n_\varepsilon})f({c_\varepsilon})$$
and then we rewrite ${c_\varepsilon}$ by the variation-of-constant formula
\begin{eqnarray}
{c_\varepsilon}(t)=e^{t(\Delta-1)}c_{0,\varepsilon}+\int^t_0 e^{(t-s)(\Delta-1)}({c_\varepsilon}-{F_\varepsilon}({n_\varepsilon})f({c_\varepsilon})-{u_\varepsilon}\cdot\nabla{c_\varepsilon})(s)ds\nonumber
\end{eqnarray}
for all $t\in(0,T_{max,\varepsilon})$. 
For $p\in(1,\infty)$, let $A := A_p$ denote the sectorial operator defined by
$$A_pu:=-\Delta u\ \ {\rm for}\ \ u\in D(A_p):=\left\{\varphi\in W^{2,p}(\Omega);\frac{\partial \varphi}{\partial \nu}=0\ \ {\rm{on}}\ \partial \Omega\right\}.$$
Then the operator $(A+1)$ possesses fractional powers $(A + 1)^\beta$, $\beta\geq0$, the domains of
which have the embedding propertie
$D((A_4+1)^\theta)\hookrightarrow W^{1,\infty}(\Omega)$ for $\theta\in(\frac78,1)$ (cf.\cite{Winkler-MN-2010,Zhang&Li-ZAMP-2015-G}). Hence, by virtue of $L^p-L^q$ estimates associated heat semigroup
(cf.\cite{Winkler-MN-2010}),  (\ref{eq-c-maxmum-principle-s0}), (\ref{eq-chi-f-phi}), (\ref{eq-F-eq1}), (\ref{eq-jfo-zzh-n-4-norm}), (\ref{eq-jfo-zzh-u-infty-norm})
and (\ref{eq-jfo-nabla-zzh-c-4-norm}), there exist positive constants $C_9$, $C_{10}$, $C_{11}$ and $C_{12}$ such that
\begin{eqnarray}
\|{c_\varepsilon(\cdot,t)}\|_{W^{1,\infty}(\Omega)}&\leq&\|(-\Delta+1)^\theta{c_\varepsilon}(\cdot,t)\|_{L^4(\Omega)}\nonumber\\
&\leq&C_9 t^{-\theta}e^{-\nu t}\|c_{0,\varepsilon}\|_{L^4(\Omega)}\nonumber\\
&\ &+C_9\int^t_0 (t-s)^{-\theta}e^{-\nu (t-s)}\|({c_\varepsilon}-{F_\varepsilon}({n_\varepsilon})f({c_\varepsilon})-{u_\varepsilon}\cdot\nabla{c_\varepsilon})(s)\|_{L^4(\Omega)}ds\nonumber\\
&\leq&C_9 \tau^{-\theta}e^{-\nu t}\|c_{0,\varepsilon}\|_{L^4(\Omega)}+C_{10}\int^t_0 (t-s)^{-\theta}e^{-\nu (t-s)}ds\nonumber\\
&\ &+C_{10}\int^t_0 (t-s)^{-\theta}e^{-\nu (t-s)}\|{n_\varepsilon}(s)\|_{L^4(\Omega)}ds\nonumber\\
&\ &+C_{10}\int^t_0 (t-s)^{-\theta}e^{-\nu (t-s)}\|\nabla {c_\varepsilon}\|_{L^4(\Omega)}ds\nonumber\\
&\leq&C_9 \tau^{-\theta}e^{-\nu t}\|c_{0,\varepsilon}\|_{L^4(\Omega)}+C_{11}\Gamma(1-\theta)\nonumber\\
&\leq&C_{12}\ \ \ \ {\rm{for\ all}}\ t\in(\tau,T_{max,\varepsilon})\label{eq-jfo-zzh-c-w1infty-norm}
\end{eqnarray}
with some $\nu>0$, where $\Gamma(\cdot)$ is the Gamma function.

We finally use (\ref{eq-jfo-zzh-c-w1infty-norm}) to estimate $\|n_\varepsilon\|_{L^\infty(\Omega)}$. For any ${\beta}>1$, from (\ref{eq-jfo-zzh-c-w1infty-norm}) we know that there exist positive constants $C_{13}$, $C_{14}$ and $C_{15}$ such that
\begin{eqnarray}
&&\frac d{dt}\int_\Omega{n_\varepsilon}^{\beta} +{\beta}({\beta}-1)\int_\Omega{n_\varepsilon}^{{\beta}-2}|\nabla{n_\varepsilon}|^p\nonumber\\
&\leq&\frac d{dt}\int_\Omega{n_\varepsilon}^{\beta} +{\beta}({\beta}-1)\int_\Omega(|\nabla{n_\varepsilon}|^2+\varepsilon)^{\frac{p-2}2}{n_\varepsilon}^{{\beta}-2}|\nabla{n_\varepsilon}|^2\nonumber\\
&=&{\beta}({\beta}-1)\int_\Omega {n_\varepsilon}{F_\varepsilon}'({n_\varepsilon})\chi({c_\varepsilon})\nabla {c_\varepsilon}{n_\varepsilon}^{{\beta}-2}\nabla{n_\varepsilon}\nonumber\\
&\leq&C_{13}\int_\Omega{n_\varepsilon}^{{\beta}-2}|\nabla {n_\varepsilon}|\nonumber\\
&\leq&\int_\Omega{n_\varepsilon}^{{\beta}-2}({\beta}({\beta}-1)|\nabla {n_\varepsilon}|^p+C_{14})\nonumber\\
&\leq&{\beta}({\beta}-1)\int_\Omega{n_\varepsilon}^{{\beta}-2}|\nabla{n_\varepsilon}|^p+\int_\Omega{n_\varepsilon}^{\beta}+C_{15}\ \ \ \ {\rm{for\ all}}\ t\in(\tau,T_{max,\varepsilon}).\label{eqsssss}
\end{eqnarray}
Therefore, integrating (\ref{eqsssss}) with respect to $t$ yields $C_{16}>0$ such that
$$\int_\Omega{n_\varepsilon}^{\beta}\leq C_{16}\ \ \ \ {\rm{for\ all}}\ t\in(\tau,T_{max,\varepsilon})$$
with any $\beta\geq 1$. Upon an application of the well-known Moser-Alikakos iteration procedure \cite{Alikakos-CPDE-1979,Tao&Winkler-JDE-2012}, we see that
\begin{eqnarray}
\|{n_\varepsilon}\|_{L^\infty(\Omega)}\leq C_{17}\ \ \ \ {\rm{for\ all}}\ t\in(\tau,T_{max,\varepsilon})\label{eq-jfo-zzh-n-infty-norm}
\end{eqnarray}
with some $C_{17}>0$. In view of (\ref{eq-jfo-zzh-u-infty-norm}), (\ref{eq-jfo-zzh-c-w1infty-norm})and (\ref{eq-jfo-zzh-n-infty-norm}),
we apply Lemma \ref{lem2.1} to reach a contradiction.
\end{proof}

\section{Further uniform estimates for (\ref{model-approximate})}
As mentioned in Remark \ref{remone}, it's crucial to prove (\ref{eqkksss}). The following lemma plays a key role in the proof.
\begin{lem}\label{lem5.1}
Let $m_0\geq1$, $p>\frac{25}{12}$, $m\geq1$ be such that
\be
\frac{m_0 (3 p-4)}{4 (p-1)}+\frac{p-2}{p-1}<m\leq m_0(p-\frac43)+3(p-2),\label{m-constant-condition-lem6.1}
\ee
and $\int_\Omega n_{0\varepsilon}^m\leq C$.
Then for all $K>0$ there exist $C=C(m,p,K)>0$ such that if for all $\varepsilon\in(0,1)$ there hold
\be
\int_\Omega{n_\varepsilon}^{m_0}(\cdot,t)\leq K\ \ \ \ {for\ all}\ t\geq0\label{n-m-0-condition-lem6.1}
\ee
and
\be
\int^t_0\int_\Omega|\nabla{c_\varepsilon}(\cdot,s)|^4ds\leq K(t+1)\ \ \ \ {for\ all}\ t\geq0,\label{nabla-c-4-condition-lem6.1}
\ee
then we have
\be
\int_\Omega{n_\varepsilon}^m(\cdot,t)\leq C(t+1)\ \ \ \ {for\ all}\ t\geq0\label{n-m-result-lem6.1}
\ee
and
\be
\int^t_0\int_\Omega|\nabla{n_\varepsilon}^{m_*}|^p\leq C(t+1)\ \ \ \ {for\ all}\ t\geq0,\label{nabla-n-m-2-result-lem6.1}
\ee
where $m_*=\frac{m-2}p+1$.
\end{lem}
\begin{proof} Test the first equation in (\ref{model-approximate}) by ${n_\varepsilon}^{m-1}$ and use Young's inequality along with (\ref{eq-chi-f-phi}), (\ref{eq-F-eq1}) and (\ref{eq-c-maxmum-principle-s0}) to see that for all $t > 0$,
\begin{eqnarray*}
\frac1m\frac{d}{dt}\int_\Omega{n_\varepsilon}^m+(m-1)\int_\Omega|\nabla{n_\varepsilon}|^p{n_\varepsilon}^{m-2}&\leq&C_1(m-1)\int_\Omega{n_\varepsilon}^{m-1}|\nabla{c_\varepsilon}|\cdot |\nabla{n_\varepsilon}|\\
&\leq&(m-1)\int_\Omega\left(\frac12|\nabla{n_\varepsilon}|^p{n_\varepsilon}^{m-2}+C_2({n_\varepsilon}^{\frac{2-m}p+m-1}|\nabla{c_\varepsilon}|)^{p'}\right)\\
\end{eqnarray*}
so that
\begin{eqnarray}
\frac 1m\frac{d}{dt}\int_\Omega{n_\varepsilon}^m+
\frac{m-1}2\int_\Omega|\nabla{n_\varepsilon}|^p{n_\varepsilon}^{m-2}\leq C_2(m-1)\int_\Omega({n_\varepsilon}^{\frac{2-m}p+m-1}|\nabla{c_\varepsilon}|)^{p'}\ \ \ \ {\rm{for\ all}}\ t>0,\label{eq1-lem6.1-proof}
\end{eqnarray}
where $p'=\frac{p}{p-1}$. Let
$${\beta}=\left({\frac{2-m}p+m-1}\right)\cdot p'=m+\frac{1}{p-1}-1$$
and
$$m_*=\frac{m-2}p+1.$$
Noticing that
$$\int_\Omega|\nabla{n_\varepsilon}|^p{n_\varepsilon}^{m-2}=\int_\Omega|{n_\varepsilon}^{\frac{m-2}p}\nabla{n_\varepsilon}|^p=\frac 1{m_*^p}\int_\Omega|\nabla{n_\varepsilon}^{m_*}|^p\ \ \ \ {\rm{for\ all}}\ t>0,$$
(\ref{eq1-lem6.1-proof}) turns into
\begin{eqnarray}
\frac 1m\frac{d}{dt}\int_\Omega{n_\varepsilon}^m+\frac{(m-1)}{2m_*^p}\int_\Omega|\nabla{n_\varepsilon}^{m_*}|^p
&\leq& C_2(m-1)\int_\Omega{n_\varepsilon}^{{\beta}}|\nabla{c_\varepsilon}|^{p'}\ \ \ \ {\rm{for\ all}}\ t>0.\label{eq2-lem6.1-proof}
\end{eqnarray}
Now in order to further estimate the right-hand side herein, we invoke the
H${\rm{\ddot{o}}}$lder's inequality to obtain
\be
\int_\Omega{n_\varepsilon}^{{\beta}}|\nabla{c_\varepsilon}|^{p'}\leq (\int_\Omega{n_\varepsilon}^{{\beta}\alpha})^{\frac1\alpha}(\int_\Omega|\nabla{c_\varepsilon}|^{4})^{\frac1{\alpha'}}\ \ \ \ {\rm{for\ all}}\ t>0,\label{eq3-lem6.1-proof}
\ee
with $\alpha=\frac{4(p-1)}{3p-4}$, and $\alpha'=\frac\alpha{\alpha-1}=\frac{4(p-1)}p$.

Due to the left part of our assumption (\ref{m-constant-condition-lem6.1}) we have
$$\frac{m_0}{m_*}-\frac{{\beta}\alpha}{m_*}=-\frac{p (4 m (p-1)+m_0 (4-3 p)-4 p+8)}{(3 p-4) (m+p-2)}<0.$$
And from $$\frac{m_*}{\beta\alpha}-\left(\frac1p-\frac13\right)=\frac{m (4 p-7)+5 (p-2)}{12 (m (p-1)-p+2)}\geq0$$
we know that $W^{1,p}(\Omega)\hookrightarrow L^{\frac{{\beta}\alpha}{m_*}}(\Omega)\hookrightarrow L^\frac{m_0}{m_*}(\Omega)$, whence in particular the number
\begin{eqnarray}
\theta:=\frac{3 (m+p-2) (4 m (p-1)+m_0 (4-3 p)-4 p+8)}{4 (m (p-1)-p+2) (3 m+(m_0+3) p-3 (m_0+2))}\label{eqjjjkkk}
\end{eqnarray}
satisfies $\theta>0$, since the left part of our assumption (\ref{m-constant-condition-lem6.1}) ensures
\begin{eqnarray}
m- \left(m_0 (1-\frac{p}{3})-p+2\right)&>&\frac{m_0 (3 p-4)}{4 p-4}+\frac{p-2}{p-1}-\left(m_0 (1-\frac{p}{3})-p+2\right)\nonumber\\
&=&\frac{p (m_0 (4 p-7)+12 (p-2))}{12 (p-1)}>0,
\end{eqnarray}
which is enough to warrant that
$$3 m+(m_0+3) p-3 (m_0+2)>0.$$
And also $\theta<1$ since
$$\theta-1=-\frac{m_0 p (5(p-2)+(4p-7)m)}{4 (m (p-1)-p+2) (3 m+(m_0+3) p-3 (m_0+2))}<0,$$
and accordingly the Gagliardo-Nirenberg inequality provides $C_3 > 0$ such that
\begin{eqnarray*}
(\int_\Omega{n_\varepsilon}^{{\beta}\alpha})^{\frac1\alpha}&=&\|{n_\varepsilon}^{m_*}\|_{\frac{{\beta}\alpha}{m_*}}^{\frac{{\beta}}{m_*}}\\
&\leq& C_3\|\nabla{n_\varepsilon}^{m_*}\|_p^{\frac{{\beta}}{m_*}\theta}\|{n_\varepsilon}^{m_*}\|_{\frac{m_0}{m_*}}^{\frac{{\beta}}{m_*}(1-\theta)}+C_3\|{n_\varepsilon}^{m_*}\|_{\frac{m_0}{m_*}}^{\frac{{\beta}}{m_*}}\ \ \ \ {\rm{for\ all}}\ t>0,
\end{eqnarray*}
where $\theta$ given by (\ref{eqjjjkkk}) satisfies
$$\frac{m_*}{{\beta}\alpha}=\theta(\frac1p-\frac13)+(1-\theta)\frac{m_*}{m_0}.$$
As $$\|{n_\varepsilon}^{m_*}\|_{\frac{m_0}{m_*}}^{\frac{m_0}{m_*}}=\int_\Omega{n_\varepsilon}^{m_0}\leq K\ \ \ \ {\rm{for\ all}}\ t>0$$
by (\ref{n-m-0-condition-lem6.1}), together with (\ref{eq3-lem6.1-proof}), $C_p$ inequality, and Young's inequality this shows that we can find $C_4 > 0$ fulfilling
\begin{eqnarray*}
\int_\Omega{n_\varepsilon}^{{\beta}}|\nabla{c_\varepsilon}|^{p'}&\leq& C_4(\|\nabla{n_\varepsilon}^{m_*}\|_p^{\frac{{\beta}}{m_*}\theta}+1)(\int_\Omega|\nabla{c_\varepsilon}|^{4})^{\frac1{\alpha'}}\\
&\leq&\delta (\|\nabla{n_\varepsilon}^{m_*}\|_p^{\frac{{\beta}}{m_*}\theta\alpha}+1)+C_\delta\int_\Omega|\nabla{c_\varepsilon}|^{4}\ \ \ \ {\rm{for\ all}}\ t>0,
\end{eqnarray*}
where $\delta>0$ and $C_\delta$ are two constant to be chosen later. From (\ref{m-constant-condition-lem6.1}) we have
$$\frac{{\beta}}{m_*}\theta\alpha-p=\frac{p^2 (3 m-m_0(3p-4)-9(p-2))}{(3 p-4) (3 m-m_0+(m_0+3)(p-2))}\leq0,$$
whence another application of Young's inequality and an appropriate choice of $\delta$ yield $C_5>0$ satisfying
\begin{eqnarray*}
C_2(m-1)\int_\Omega{n_\varepsilon}^{{\beta}}|\nabla{c_\varepsilon}|^{p'}\leq \frac{(m-1)}{4m_*^p}\int_\Omega|\nabla{n_\varepsilon}^{m_*}|^p+C_5\int_\Omega|\nabla{c_\varepsilon}|^{4}+C_5\ \ \ \ {\rm{for\ all}}\ t>0.
\end{eqnarray*}
This shows that (\ref{eq2-lem6.1-proof}) implies that
\begin{eqnarray}\label{eq2-lem6.2}
\frac 1m\frac{d}{dt}\int_\Omega{n_\varepsilon}^m+\frac{(m-1)}{4m_*^p}\int_\Omega|\nabla{n_\varepsilon}^{m_*}|^p
 \leq C_5\int_\Omega|\nabla{c_\varepsilon}|^{4}+C_5\ \ \ \ {\rm{for\ all}}\ t>0.
\end{eqnarray}
Let $y(t):=\int_\Omega{n_\varepsilon}^m(\cdot,t), t>0$, and $h(t):=C_5\int_\Omega|\nabla{c_\varepsilon}|^{4}+C_5$, $t>0$, then in view of the nonnegativity of $h(t)$ and our assumption (\ref{nabla-c-4-condition-lem6.1}), (\ref{eq2-lem6.2}) thus firstly implies that
\be
y(t)\leq y(0)+m\int^t_0h(s)ds\leq  y(0)+C_5Km(t+1)+C_5t\ \ \ \ {\rm{for\ all}}\ t>0,\label{eq6-lem6.1-proof}
\ee
whereupon (\ref{nabla-c-4-condition-lem6.1}), (\ref{eq2-lem6.2}) and (\ref{eq6-lem6.1-proof}) secondly entail that
\begin{eqnarray*}
\int^t_0\int_\Omega|\nabla{n_\varepsilon}^{m_*}|^p&\leq&\frac{4m_*^p}{m-1}\int^t_0h(s)ds+\frac{4m_*^p}{(m-1)m}y(0)\\
&\leq& \frac{4m_*^p}{m-1}(C_5K(t+1)+C_5t)+\frac{4m_*^p}{(m-1)m}y(0)\ \ \ \ {\rm{for\ all}}\ t>0,
\end{eqnarray*}
so that indeed both (\ref{n-m-result-lem6.1}) and (\ref{nabla-n-m-2-result-lem6.1}) hold with some conveniently large $C=C(m,p,K)>0$.
\end{proof}

\begin{rem} We shall point out that there exist $m\geq1$ satisfies (\ref{m-constant-condition-lem6.1}). Indeed, a formal computation shows that
$${m_0}\left(p-\frac{4}{3}\right)+3 (p-2)-\left(\frac{{m_0} (3 p-4)}{4 p-4}+\frac{p-2}{p-1}\right)=\frac{(3 p-4) ({m_0} (4 p-7)+12 (p-2))}{12 (p-1)}>0.$$
and
$${m_0}\left(p-\frac{4}{3}\right)+3 (p-2)-1=m_0 \left(p-\frac{4}{3}\right)+3 p-7\geq 4p-\frac{25}3>0,$$
since $p>\frac{25}{12}.$
\end{rem}

We shall see that this indeed leads to improved information whenever $p >\frac{32}{15}>\frac{25}{12}$. Repeatedly applying Lemma \ref{lem5.1}, we can gradually raise the parameter $m$ up to 2, which is equivalent to $m_*=1$, so we can get a higher order gradient estimate of ${n_\varepsilon}$.

 \begin{cor}\label{cor5.2}
Let $p=\frac{32}{15}+\delta$ for some $\delta>0$ small such that $\delta_1=\frac{\log \left(\frac{25 \delta}{20 \delta+1}\right)}{\log \left(\delta+\frac{4}{5}\right)}$ is a positive integer and $\int_\Omega n_{0\varepsilon}^{2}\leq C$.
Then for all $K>0$ there exist $C=C(m,p,K)>0$ such that if for some $\varepsilon\in(0,1)$ we have
\be
\int_\Omega{n_\varepsilon}(\cdot,t)\leq K\ \ \ \ {\rm{for\ all}}\ t\geq0\nonumber
\ee
and
\be
\int^t_0\int_\Omega|\nabla{c_\varepsilon}(\cdot,s)|^4ds\leq K(t+1)\ \ \ \ {\rm{for\ all}}\ t\geq0,\nonumber
\ee
then
\be
\int^T_0\int_\Omega|\nabla{n_\varepsilon}|^p\leq C(T+1)\ \ \ \ {\rm{for\ all}}\ t\geq0.\label{eq-cor5.0-nabla-n-p}
\ee
\end{cor}

 \begin{proof}
 We define $(m_k)_{k\in \mathbb{N}_0}\subset \mathbb{R}$ by letting $m_0 = 1$ and
 $$m_{k+1}=m_k(\frac{32}{15}+\delta-\frac43)+3(\frac{32}{15}+\delta-2)=m_k(\delta+\frac45)+3(\delta+\frac{2}{15})\ \ \ \ {\rm{for}}\ k\in \mathbb{N}_0.$$
 A simple computation shows that
 \be
 m_k=(\delta+\frac45)^k(1+\frac{15\delta+2}{5\delta-1})-\frac{15\delta+2}{5\delta-1}\ \ \ \ {\rm{for}}\ k\in \mathbb{N}.
 \ee
Solving the equation
 \be
 (\delta+\frac45)^x(1+\frac{15\delta+2}{5\delta-1})-\frac{15\delta+2}{5\delta-1}=2\ \ \ \ {\rm{for}}\ x\geq0,
 \ee
we have
 \be
 x=\frac{\log \left(\frac{25 \delta}{20 \delta+1}\right)}{\log \left(\delta+\frac{4}{5}\right)}>0.
 \ee
 Since $\delta_1=\delta_1(\delta)$ is a continuous and monotonically decreasing real function on the interval $(0,\frac1{10})$, and $\lim_{\delta\rightarrow0}\delta_1=+\infty$ as well as ${\delta_1}|_{\delta=\frac1{100}}=\frac{\log \left(\frac{24}{5}\right)}{\log \left(\frac{100}{81}\right)}\approx7.44$.
 Hence we can take $\delta>0$ small such that $\delta_1$ is a positave integer, and apply Lemma \ref{lem5.1} $\delta_1$ times, then we obtain $m_{\delta_1}=2$ hence $m_*=1$ and get the desired result.
 \end{proof}
A direct application of Corollary \ref{cor5.2} enable us to get a higher order regularity of $n_\varepsilon$.
\begin{lem}\label{lem5.3}
 There exists $C>0$ such that for any $\varepsilon\in(0,1)$, we  have
\begin{eqnarray}
\int^T_0\int_\Omega |{u_\varepsilon}|^{\frac{10}3}\leq C(T+1)\ \ \ \ {for\ all}\ T>0,\label{eq-lem5.1-u}
\end{eqnarray}
and
\begin{eqnarray}
\int^T_0\int_\Omega {n_\varepsilon^r}\leq C(T+1),\ \ \ \ {for\ all}\ T>0,\label{eq-lem5.1-n}
\end{eqnarray}
where
\begin{eqnarray}\label{parameter-r}
  \left\{\begin{array}{lll}
     \medskip
     r\geq1,&\ \ p\geq3,\\
     \medskip
     r\in[1,\frac{3p}{3-p}),&\ \ p\in(\frac{32}{15},3).
  \end{array}\right.
\end{eqnarray}
\end{lem}

\begin{proof}
(i) Let $\theta=\frac{3 p (r-1)}{(4 p-3) r}$, then $\theta\in[0,1)$ satisfies
\be
\frac1r=\theta(\frac1p-\frac13)+(1-\theta).\label{eq5.1}
\ee
Thus, invoking the Gagliardo-Nirenberg inequality along with (\ref{approximate-initialdata-n}), Corollary \ref{cor5.2} and Young's inequality
we obtain $C_1>0$ and $C_2>0$ such that
\begin{eqnarray}
\|{n_\varepsilon}\|_{L^r((0,T)\times\Omega)}&\leq&C_1(\|\nabla{n_\varepsilon}\|_{_{L^p((0,T)\times\Omega)}}^\theta\|{n_\varepsilon}\|^{(1-\theta)}_{_{L^1((0,T)\times\Omega)}}+\|{n_\varepsilon}\|_{_{L^1((0,T)\times\Omega)}})\\
&\leq& C(T+1)\ \ \ \ {\rm{for\ all}}\ T>0.
\end{eqnarray}

(ii) As a particular consequence of Lemma \ref{lem3.3}, we can find $C_3>0$ and $C_4>0$ such that
\begin{eqnarray}
\int_\Omega|{u_\varepsilon}(\cdot,t)|^2\leq C_3\ \ \ \ {\rm{for\ all}}\ t>0\ \ \ \ {\rm{and}}\ \ \ \ \int^T_0\int_\Omega|\nabla{u_\varepsilon}|^2\leq C_4(T+1)\ \ \ \ {\rm{for\ all}}\ T>0.\nonumber
\end{eqnarray}
Upon an application of the Gagliardo-Nirenberg inequality together with Poincar${\rm{\acute{e}}}$ inequality we see that with some $C_5>0$ and $C_6>0$ we have
\begin{eqnarray}
\int^T_0\int_\Omega|{u_\varepsilon}|^{\frac{10}3}&=&\int^T_0\|{u_\varepsilon}\|^{\frac{10}3}_{\frac{10}3}\nonumber\\
&\leq&C_5\int^T_0(\|{u_\varepsilon}\|^{\frac35}_{W^{1,2}(\Omega)}\|{u_\varepsilon}\|^{\frac25}_2)^{\frac{10}3}\nonumber\\
&\leq&C_6\int^T_0\|\nabla{u_\varepsilon}\|^2_2\nonumber\\
&\leq&C_6C_4(T+1)\ \ \ \ {\rm{for\ all}}\ T>0,\nonumber
\end{eqnarray}
whereby the proof is completed.
\end{proof}

Since $3>p\geq2$ implies $\frac{3p}{3-p}\geq6$, we immediately obtain:
\begin{cor}\label{cor5.4}
 There exists $C>0$ such that for $\varepsilon\in(0,1)$ we have
\begin{eqnarray}
\int^T_0\int_\Omega {n_\varepsilon}^{6}\leq C(T+1),\ \ \ \ {for\ all}\ T>0.\label{eq-cor5.1}
\end{eqnarray}
\end{cor}

The following estimates concerning the time derivative can allow us to apply Aubin-Lions lemma later to derive strong compactness properties.
\begin{lem}\label{lem5.5}
There exists $C>0$ such that for all $\varepsilon\in(0,1)$ we have
\begin{eqnarray}
\int^T_0\|{\partial_t}{n_\varepsilon}(\cdot,t)\|^{p'}_{(W^{1,p}(\Omega))^*}\leq C(T+1)\ \ \ \ {for\ all}\ T>0,\label{eq1-lem5.2}\\
\int^T_0\|{\partial_t}{c_\varepsilon}(\cdot,t)\|^{\frac{10}3}_{(W^{1,\frac{10}7}(\Omega))^*}\leq C(T+1)\ \ \ \ {for\ all}\ T>0,\label{eq2-lem5.2}\\
\int^T_0\|{\partial_t}{u_\varepsilon}(\cdot,t)\|^{\frac{5}{4}}_{(W^{1,5}(\Omega))^*}\leq C(T+1)\ \ \ \ {for\ all}\ T>0,\label{eq3-lem5.2}
\end{eqnarray}
where $p'=\frac p{p-1}$.
\end{lem}

\begin{proof}
(i)For arbitrary $t>0$ and $\varphi\in C^\infty(\bar{\Omega})$, multiplying the first equation in (\ref{model-approximate}) by $\varphi$,
integrating by parts and using the H${\rm{\ddot{o}}}$lder's inequality we obtain
\begin{eqnarray}
&&|\int_\Omega{\partial_t}{n_\varepsilon}(\cdot,t)\varphi|\nonumber\\
&=&|-\int_\Omega(|\nabla{n_\varepsilon}|^2+\varepsilon)^\frac{p-2}2\nabla{n_\varepsilon}\cdot\nabla\varphi+\int_\Omega{n_\varepsilon}{F_\varepsilon}'({n_\varepsilon})\chi({c_\varepsilon})\nabla{c_\varepsilon}\cdot\nabla\varphi+\int_\Omega{n_\varepsilon}{u_\varepsilon}\cdot\nabla\varphi|\nonumber\\
&\leq&\int_\Omega(|\nabla{n_\varepsilon}|^2+\varepsilon)^\frac{p-2}2|\nabla{n_\varepsilon}|\cdot|\nabla\varphi|+C\|{n_\varepsilon}\nabla{c_\varepsilon}\|_{p'}\|\nabla\varphi\|_{p}+\|{n_\varepsilon}{u_\varepsilon}\|_{p'}\|\nabla\varphi\|_{p}\nonumber\\
&\leq&C_1(\|(|\nabla{n_\varepsilon}|^2+\varepsilon)^\frac{p-1}2\|_{p'}+\|{n_\varepsilon}\nabla{c_\varepsilon}\|_{p'}+\|{n_\varepsilon}{u_\varepsilon}\|_{p'})\|\nabla\varphi\|_{p}\nonumber
\end{eqnarray}
with some $C_1>0$.
By lemma \ref{lem3.3} and (\ref{eq-c-maxmum-principle-s0}) we have
\begin{eqnarray}
\int^T_0\int_\Omega|\nabla{c_\varepsilon}|^4=\int^T_0\int_\Omega\frac{|\nabla{c_\varepsilon}|^4}{{c_\varepsilon}^3}{c_\varepsilon}^3\leq C_2(T+1)\ \ \ \ {\rm{for\ all}}\ T>0\label{eq-lem5.2-zzh-c-4-norm}
\end{eqnarray}
with some $C_2>0$. Notice that $\frac16+\frac14<\frac16+\frac3{10}<\frac1{p'}$. Thus, invoking (\ref{eq-lem5.2-zzh-c-4-norm}) along with
(\ref{eq-cor5.0-nabla-n-p}), (\ref{eq-cor5.1}), Lemma \ref{lem3.3}, H${\rm{\ddot{o}}}$lder's inequality and Young's inequality
we obtain positive constants $C_3$, $C_4$, and $C_5$ such that
\begin{eqnarray}
&&\int^T_0\|{\partial_t}{n_\varepsilon}(\cdot,t)\|^{p'}_{(W^{1,p}(\Omega))^*}\nonumber\\
&\leq& C_3(\int^T_0\int_\Omega(|\nabla{n_\varepsilon}|^2+1)^{\frac{p-1}2{p'}}+\int^T_0\int_\Omega|{n_\varepsilon}\nabla{c_\varepsilon}|^{p'}+\int^T_0\int_\Omega|{n_\varepsilon}{u_\varepsilon}|^{p'})\nonumber\\
&\leq&C_4(\int^T_0\int_\Omega(|\nabla{n_\varepsilon}|^{p}+1)+\int^T_0\int_\Omega|{n_\varepsilon}|^{6}+\int^T_0\int_\Omega|\nabla{c_\varepsilon}|^{4}\nonumber\\
&\ &+\int^T_0\int_\Omega|{n_\varepsilon}|^{6}+\int^T_0\int_\Omega|{u_\varepsilon}|^{\frac{10}3}+|\Omega|T)\nonumber\\
&\leq&C_5(T+1)\ \ \ \ {\rm{for\ all}}\ T>0.\label{eq-lem5.2-estimate-i}
\end{eqnarray}

(ii) Likewise, given any $\varphi\in C^\infty(\bar{\Omega})$, multiplying the second equation in (\ref{model-approximate}) by $\varphi$ to see that
\begin{eqnarray}
|\int_\Omega{\partial_t}{c_\varepsilon}\varphi|&=&|-\int_\Omega\nabla{c_\varepsilon}\cdot\nabla\varphi-\int_\Omega {F_\varepsilon}({n_\varepsilon})f({c_\varepsilon})\varphi+\int_\Omega{u_\varepsilon}{c_\varepsilon}\cdot\nabla\varphi|\nonumber\\
&\leq&C_6(\|\nabla{c_\varepsilon}\|_{\frac{10}3}+\|{n_\varepsilon}f({c_\varepsilon})\|_{\frac{10}3}+\|{u_\varepsilon}{c_\varepsilon}\|_{\frac{10}3})\|\varphi\|_{W^{1,\frac{10}7}}\nonumber\\
&\leq&C_7(\|\nabla{c_\varepsilon}\|_{\frac{10}3}+\|{n_\varepsilon}\|_{\frac{10}3}+\|{u_\varepsilon}\|_{\frac{10}3})\|\varphi\|_{W^{1,\frac{10}7}},\nonumber
\end{eqnarray}
with some $C_6>0$ and $C_7>0$. A combination of (\ref{eq-lem5.1-u}), (\ref{eq-cor5.1}) and (\ref{eq-lem5.2-zzh-c-4-norm}) shows that there exist positive constants $C_8$ and $C_{9}$ such that
\begin{eqnarray}
\int^T_0\|{\partial_t}{c_\varepsilon}(\cdot,t)\|^{\frac{10}3}_{(W^{1,{\frac{10}7}}(\Omega))^*}&\leq&C_8(\int^T_0\int_\Omega|\nabla{c_\varepsilon}|^{4}+\int^T_0\int_\Omega{n_\varepsilon}^{6}+\int^T_0\int_\Omega|{u_\varepsilon}|^{\frac{10}3}+|\Omega|T)\nonumber\\
&\leq&C_{9}(T+1)\ \ \ \ {\rm{for\ all}}\ T>0.\label{eq-lem5.2-estimate-ii}
\end{eqnarray}
(iii) (see also \cite[Lemma 3.11]{Tao&Winkler-AIHP-2013}) Given $\varphi\in C^\infty_{0,\sigma}(\Omega;\mathbb{R}^3)$, we infer form the third equation in (\ref{model-approximate}) that there exist positive constants $C_{10}$ and $C_{11}$ such that
\begin{eqnarray}
&&\int\|\partial_t u_\varepsilon(\cdot,t)\|^\frac54_{(W^{1,5}_{0,\sigma}(\Omega))^*}dt\\
&\leq&C_{10}\int^T_0\int_\Omega|\nabla u_\varepsilon|^\frac54+C_{10}\int^T_0\int_\Omega
|Y_\varepsilon u_\varepsilon\otimes u_\varepsilon|^\frac54+C_{10}\int^T_0\int_\Omega |n_\varepsilon\nabla\Phi|^\frac54\\
&\leq&C_{11}\int^T_0\int_\Omega|\nabla u_\varepsilon|^2+C_{11}\int^T_0\int_\Omega
|Y_\varepsilon u_\varepsilon|^2+C_{11}\int^T_0\int_\Omega u_\varepsilon^\frac{10}3\\
&&+C_{11}\int^T_0\int_\Omega n_\varepsilon^6+C_{11}|\Omega|T\ \ \ \ {\rm{for\ all}}\ T>0,
\end{eqnarray}
where we use Young's inequality and $\nabla\Phi\in L^\infty(\Omega)$. Since $\|Y_\varepsilon v\|_{L^2(\Omega)}\leq \|v\|_{L^2(\Omega)}$ for all $v\in L^2_\sigma(\Omega)$ and hence $\int^T_0\int_\Omega|Y_\varepsilon u_\varepsilon|^2\leq \int^T_0\int_\Omega|u_\varepsilon|^\frac{10}3+|\Omega|T$ for all $T>0$, (\ref{eq3-lem5.2}) results from this upon another application of Lemma \ref{lem3.3}, Lemma \ref{lem5.3} and Corollary \ref{cor5.4}.
\end{proof}

\section{Passing to the limit. Proof of Theorem \ref{main result}}
In this section we construct global weak solutions for (\ref{CNS}), (\ref{eq-ic}) and (\ref{eq-bc}). Before going into detail, let us first give the definition of weak solution.
\begin{defn}\label{defnaaaa}
We call $(n, c, u)$ a {\it{global weak solution}} of (\ref{CNS}), (\ref{eq-ic}) and (\ref{eq-bc}) if
\begin{equation*}
n\in L_{loc}^{1}(\bar{\Omega}\times[0, \infty)),\quad c\in L_{loc}^1([0, \infty);W^{1,1}(\Omega)),\quad u\in\big{(}L_{loc}^1([0, \infty);W_0^{1,1}(\Omega))\big{)}^3
\end{equation*}
such that $n\geq0$ and $c\geq0$ {\it a.e.} in $\Omega\times(0, \infty)$, and that
\begin{eqnarray*}
&&nf(c)\in L_{loc}^1([0, \infty);L^{1}(\Omega)),\\
&&|\nabla n|^{p-2}\nabla n,\ n\chi(c)\nabla c,\ nu\ \mbox{and}\ cu\ \mbox{belong to}\left(L_{loc}^1([0, \infty);L^{1}(\Omega))\right)^3,\\
&&u\otimes u\in \left(L_{loc}^1([0, \infty);L^{1}(\Omega))\right)^{3\times3}
\end{eqnarray*}
and that
\begin{eqnarray}
&&\int^\infty_0\int_\Omega n_t\phi_1-\int^\infty_0\int_\Omega nu\cdot\nabla\phi_1=-\int^\infty_0\int_\Omega |\nabla n|^{p-2}\nabla n\cdot\nabla\phi_1+\int^\infty_0\int_\Omega n\chi(c)\nabla c\cdot\nabla\phi_1,\nonumber\\
&&\int^\infty_0\int_\Omega c_t\phi_2-\int^\infty_0\int_\Omega cu\cdot\nabla\phi_2=-\int^\infty_0\int_\Omega \nabla c\cdot\nabla\phi_2-\int^\infty_0\int_\Omega nf(c)\phi_2,\nonumber\\
&&\int^\infty_0\int_\Omega u_t\cdot\phi_3-\int^\infty_0\int_\Omega u\otimes u\cdot\nabla\phi_3=-\int^\infty_0\int_\Omega \nabla u\cdot\nabla\phi_3+\int^\infty_0\int_\Omega n\nabla\Phi\cdot\phi_3,\nonumber
\end{eqnarray}
hold for all $\phi_1,\phi_2\in C^\infty_0(\Omega\times[0,\infty))$ and $\phi_3\in(C^\infty_0(\Omega\times[0,\infty)))^3$ satisfying $\nabla\cdot\phi_3=0$.
\end{defn}

We need the following auxiliary lemma before proving Theorem \ref{main result}.
\begin{lem}\label{lem6.1} Let $\Omega$ be a bounded domain of $\mathbb{R}^d$ with $d\geq1$ and $f_k\rightharpoonup f$ in $L^p(\Omega)$ with $p\in(1,\infty)$, if there hold
\begin{equation}
f_k\rightarrow f\ \ a.e.\ \ {\rm{in}}\ \ \Omega,\ \ \ \ as\ \ k\rightarrow\infty,\label{eq-ae}
\end{equation}
then there exists a subsequence $k=\{k_j\}(j=1,2,3,\cdots)$ such that
\begin{equation}
f_{k_j}\rightarrow f\ \ {\rm{in}}\ \ L^q(\Omega),\ \ \ \ as\ \ k_j\rightarrow\infty
\end{equation}
for any $1\leq q< p$.
\end{lem}
\begin{proof} For any $1\leq q< p$, we have $|f_k|^q$ uniformly bounded in $L^{\frac pq}(\Omega)$, hence there exists a subsequence $(k_j)_{j\in \mathbb{N}}$ such that $|f_{k_j}|^q\rightharpoonup |f|^q$ in $L^{\frac pq}(\Omega)$. This proves that
$$\lim_{{k_j}\rightarrow\infty}\|f_{k_j}\|_q^q=\lim_{{k_j}\rightarrow\infty}\int_\Omega |f_{k_j}|^q=\lim_{{k_j}\rightarrow\infty}\int_\Omega\mathbf{ 1}_{\Omega}|f_{k_j}|^q=\int_\Omega |f|^q=\|f\|_q^q,$$
which together with (\ref{eq-ae}) is enough to complete the proof. Indeed, applying the inequality $|a-b|^q\leq2^{q-1}(|a|^q+|b|^q)$ we obtain
$$2^{q-1}(|f_{k_j}(x)|^q+|f(x)|^q)-|f_{k_j}(x)-f(x)|^q\geq0,\ \ \ \ \forall x\in\Omega.$$
In view of Fatou's lemma, we obtain
\begin{eqnarray*}
2^q\int_\Omega|f|^q&\leq&\liminf_{{k_j}\rightarrow\infty}\int_\Omega\left[2^{q-1}(|f_{k_j}|^q+|f|^q)-|f_{k_j}-f|^q\right]\\
&\leq&2^{q-1}\lim_{{k_j}\rightarrow\infty}\int_\Omega|f_{k_j}|^q+2^{q-1}\int_\Omega|f|^q+\liminf_{{k_j}\rightarrow\infty}\int_\Omega(-|f_{k_j}-f|^q)\\
&=&2^q\int_\Omega|f|^q-\limsup_{{k_j}\rightarrow\infty}\int_\Omega|f_{k_j}-f|^q,
\end{eqnarray*}
which means
$$\limsup_{{k_j}\rightarrow\infty}\int_\Omega|f_{k_j}-f|^q\leq0,$$
and this get the desired result.
\end{proof}
We can now prove our main result.

\noindent{\bf Proof of Theorem \ref{main result}.}
By Lemma \ref{lem3.3}, Corollary \ref{cor5.4}, Lemma \ref{lem5.3} and Lemma \ref{lem5.5}, for some $C>0$ independent of $\varepsilon$, there hold
\begin{eqnarray}
&&\left\|n_{\varepsilon}\right\|_{L_{loc}^{p}([0, \infty); W^{1,p}(\Omega))}\leq C(T+1),\label{slx-space-n}\\
&&\left\|\left(n_{\varepsilon}\right)_t\right\|_{L_{loc}^{p'}([0, \infty); (W^{1,p}(\Omega))^{*})}\leq C(T+1),\label{slx-space-nt}\\
&&\left\|c_{\varepsilon}\right\|_{L^{2}_{loc}([0, \infty); W^{2,2}(\Omega))}\leq C(T+1),\\
&&\left\|(c_{\varepsilon})_t\right\|_{L^{\frac{10}3}_{loc}([0, \infty); (W^{1,\frac{10}7}(\Omega))^{*})}\leq C(T+1),\\
&&\left\|u_{\varepsilon}\right\|_{L_{loc}^{2}([0, \infty); W^{1,2}(\Omega))}\leq C(T+1),\quad \mbox{and}\\
&&\left\|(u_{\varepsilon) t}\right\|_{L_{loc}^{\frac{5}{4}}([0, \infty); (W_{0,\sigma}^{1,5}(\Omega))^{*})}\leq C(T+1)
\end{eqnarray}
for all $T>0$. Therefore, the Aubin-Lions lemma \cite{J.L.Lions-DGVP-1969} asserts that
\begin{eqnarray*}
&&(n_{\varepsilon})_{\varepsilon\in(0,1)}\ \mbox{is strongly precompact in}\ L_{loc}^{p}(\bar{\Omega}\times[0, \infty)),\\
&&(c_{\varepsilon})_{\varepsilon\in(0,1)}\ \mbox{is strongly precompact in}\ L_{loc}^{2}([0, \infty); W^{1,2}(\Omega))\quad \mbox{and}\\
&&(u_{\varepsilon})_{\varepsilon\in(0,1)}\ \mbox{is strongly precompact in}\ L_{loc}^{2}(\bar{\Omega}\times[0, \infty)).
\end{eqnarray*}
This together with (\ref{eq-c-maxmum-principle-s0}), Lemma \ref{lem3.3}, Corollary \ref{cor5.2} and Corollary \ref{cor5.4}, yields a subsequence $\varepsilon:=\varepsilon_j\in(0,1)$ $(j=1,2,3,\cdots)$ and functions $n$, $c$ and $u$ such that
\begin{eqnarray}
n_{\varepsilon}\rightarrow n              &\quad& \mbox{in}\ L_{loc}^{p}(\bar{\Omega}\times[0, \infty)),\ \mbox{and a.e. in}\ \Omega\times(0, \infty),\label{n-strong-p}\\
n_{\varepsilon}\rightharpoonup n          &\quad& \mbox{in}\ L_{loc}^{r}(\bar{\Omega}\times[0, \infty)),\label{n-weak-r}\\
\nabla n_{\varepsilon}\rightharpoonup \nabla n          &\quad& \mbox{in}\ L_{loc}^{p}(\bar{\Omega}\times[0, \infty)),\label{nabla-n-weak-p}\\
|\nabla {n_\varepsilon}|^{p-2}\nabla n_{\varepsilon}\rightharpoonup \Gamma_1
                                          &\quad& \mbox{in}\ L_{loc}^{p'}(\bar{\Omega}\times[0, \infty))\label{mix-nabla-n}
\end{eqnarray}
with $\Gamma_1=|\nabla n|^{p-2}\nabla n$ which will be showed in Lemma \ref{lem6.2} and $r$ is given by Lemma \ref{lem5.3}, and
\begin{eqnarray}
c_{\varepsilon}\rightarrow c &\quad& \mbox{in}\ L_{loc}^{2}([0, \infty); W^{1,2}(\Omega))\ \mbox{and a.e. in}\ \Omega\times(0, \infty),\label{c-strong-w-one-two}\\
c_{\varepsilon}\stackrel{\ast}{\rightharpoonup} c
                                          &\quad& \mbox{in}\ L^{\infty}(\Omega\times(0, \infty)),\label{c-weak-star-infty}\\
\nabla c_{\varepsilon}\rightharpoonup \nabla c
                                          &\quad& \mbox{in}\ L_{loc}^{4}(\bar{\Omega}\times[0, \infty))\label{nabla-c-weak-four}
\end{eqnarray}
as well as
\begin{eqnarray}
u_{\varepsilon}\rightarrow u              &\quad& \mbox{in}\ L_{loc}^{2}(\bar{\Omega}\times[0, \infty))\ \mbox{and a.e. in}\ \Omega\times(0, \infty),\label{c-strong-two}\\
u_{\varepsilon}\stackrel{\ast}{\rightharpoonup} u
                                          &\quad& \mbox{in}\ L^{\infty}([0, \infty); L^2_{\sigma}(\Omega)),\label{u-weak-star-infty}\\
u_{\varepsilon}\rightharpoonup u          &\quad& \mbox{in}\ L_{loc}^{\frac{10}{3}}(\bar{\Omega}\times[0, \infty)),\label{u-weak-frac-ten-three}\\
\nabla u_{\varepsilon}\rightharpoonup \nabla u
                                          &\quad& \mbox{in}\ L_{loc}^{2}(\bar{\Omega}\times[0, \infty))\label{nabla-u-weak-two}
\end{eqnarray}
as $\varepsilon\searrow0$.

According to Lemma \ref{lem5.3} and (\ref{n-strong-p}), an application of lemma \ref{lem6.1} provides a sequence $(\varepsilon_j)_{j\in\mathbb{N}}\subset(0,1)$ and the limit function $n$ such that $\varepsilon_j\rightarrow0$ as $j\rightarrow\infty$ and such that
\begin{equation}
n_{\varepsilon}\rightarrow n              \quad \mbox{in}\ L_{loc}^{5}(\bar{\Omega}\times[0, \infty)),\ \mbox{and a.e. in}\ \Omega\times(0, \infty).\label{n-strong-7}
\end{equation}
Similarly, according to (\ref{eq-chi-f-phi}), (\ref{eq-F-eq1})-(\ref{eq-F-eq3}),(\ref{eq-c-maxmum-principle-s0}), (\ref{slx-space-n})-(\ref{nabla-u-weak-two}) and Lemma \ref{lem6.1} 
we can obtain
\begin{eqnarray}
F'_{\varepsilon}(n_{\varepsilon})\chi(c_{\varepsilon})\nabla c_{\varepsilon}\rightarrow \chi(c)\nabla c
                                                         &\quad& \mbox{in}\ L_{loc}^{\frac{10}3}(\bar{\Omega}\times[0, \infty)),\label{zuheone}\\
f(c_{\varepsilon})\rightarrow f(c)                       &\quad& \mbox{in}\ L_{loc}^{\frac{10}{3}}(\bar{\Omega}\times[0, \infty)),\\
c_{\varepsilon}\rightarrow c                             &\quad& \mbox{in}\ L_{loc}^{2}(\bar{\Omega}\times[0, \infty)),\\
F_{\varepsilon}(n_{\varepsilon})\rightarrow n            &\quad& \mbox{in}\ L_{loc}^{\frac{10}{7}}(\bar{\Omega}\times[0, \infty)),\\
Y_{\varepsilon}u_{\varepsilon}\rightarrow u              &\quad& \mbox{in}\ L_{loc}^{2}(\bar{\Omega}\times[0, \infty))
\end{eqnarray}
as $\varepsilon\searrow0$. And moreover we have
\begin{eqnarray}
n_{\varepsilon}u_{\varepsilon}\rightarrow nu         &\quad& \mbox{in}\ L_{loc}^{1}(\bar{\Omega}\times[0, \infty)),\\
n_{\varepsilon}F'_{\varepsilon}(n_{\varepsilon})\chi(c_{\varepsilon})\nabla c_{\varepsilon}\rightarrow n\chi(c)\nabla c
                                                         &\quad& \mbox{in}\ L_{loc}^{2}(\bar{\Omega}\times[0, \infty)),\label{zuhetwo}\\
F_{\varepsilon}(n_{\varepsilon})f(c_{\varepsilon})\rightarrow nf(c)
                                                         &\quad& \mbox{in}\ L_{loc}^{1}(\bar{\Omega}\times[0, \infty)),\\
c_{\varepsilon}u_{\varepsilon}\rightarrow cu             &\quad& \mbox{in}\ L_{loc}^{1}(\bar{\Omega}\times[0, \infty)),\\
Y_{\varepsilon}u_{\varepsilon}\otimes u_{\varepsilon}\rightarrow u\otimes u
                                                         &\quad& \mbox{in}\ L_{loc}^{1}(\bar{\Omega}\times[0, \infty))
\end{eqnarray}
as $\varepsilon\searrow0$.
Based on the above convergence properties, we can pass to the limit in each term of weak formulation for (\ref{model-approximate}) to construct a global weak solution of (\ref{CNS})-(\ref{eq-bc}) and thereby completes the proof.

Now it remains to show that (\ref{mix-nabla-n}) holds with $\Gamma_1=|\nabla n|^{p-2}\nabla n$. Monotonic method (cf. \cite[Section 2.1]{J.L.Lions-DGVP-1969}, see also \cite{Bendahmane2009}) is used to prove the convergence. Actually, we have
\begin{lem}\label{lem6.2} Under the assumptions of Theorem \ref{main result}, we have
$$|\nabla {n_\varepsilon}|^{p-2}\nabla n_{\varepsilon}\rightharpoonup |\nabla n|^{p-2}\nabla n\quad \mbox{in}\ L_{loc}^{p'}(\bar{\Omega}\times[0, \infty)).$$
\end{lem}
\begin{proof}We define $\Sigma_T:=\{(t,s,x):(x,s)\in \Omega\times(0,t), t\in[0,T]\}$. It's equivalent to show that
\begin{equation}
\int_{\Sigma_T}(\Gamma_1-|\nabla \sigma|^{p-2}\nabla \sigma)\cdot(\nabla n-\nabla \sigma)dxdsdt\geq 0,\ \ \forall \sigma\in L^{p}(0,T;W^{1,p}(\Omega)).\label{eq-lem6.2}
\end{equation}
For all fixed $\varepsilon>0$, we have the decomposition
$$\int_{\Sigma_T}(|\nabla {n_\varepsilon}|^{p-2}\nabla n_{\varepsilon}-|\nabla \sigma|^{p-2}\nabla \sigma)\cdot(\nabla n-\nabla \sigma)dxdsdt=I_1+I_2+I_3,$$
with
\begin{eqnarray*}
I_1&=&\int_{\Sigma_T}|\nabla {n_\varepsilon}|^{p-2}\nabla n_{\varepsilon}\cdot(\nabla n-\nabla n_\varepsilon)dxdsdt,\\
I_2&=&\int_{\Sigma_T}(|\nabla {n_\varepsilon}|^{p-2}\nabla n_{\varepsilon}-|\nabla \sigma|^{p-2}\nabla \sigma)\cdot(\nabla n_\varepsilon-\nabla \sigma)dxdsdt,\\
I_3&=&\int_{\Sigma_T}|\nabla \sigma|^{p-2}\nabla \sigma\cdot(\nabla n_\varepsilon-\nabla n)dxdsdt.
\end{eqnarray*}
Clearly, $I_2\geq C|\nabla n_\varepsilon-\nabla \sigma|^p\geq0$, where $C$ is a positive constant only depending on $p$, and from (\ref{nabla-n-weak-p}) we deduce that $I_3\rightarrow0$ as $\varepsilon\searrow0$.

For $I_1$, if we multiply the first equation of (\ref{model-approximate}) by $(n-n_\varepsilon)$ and integrate over $\Sigma_T$, we obtain
\begin{eqnarray*}
&&\int_{\Sigma_T}(|\nabla n_\varepsilon|^2+\varepsilon)^\frac{p-2}2\nabla n_{\varepsilon}\cdot(\nabla n-\nabla n_\varepsilon)\\
&=&-\int^T_0\int^t_0<\partial_t n_\varepsilon,n>dsdt+\int^T_0\int^t_0<\partial_t n_\varepsilon,n_\varepsilon>dsdt\\
&&-\int_{\Sigma_T}{u_\varepsilon}\cdot\nabla {n_\varepsilon}(n-{n_\varepsilon})dxdsdt+\int_{\Sigma_T}{n_\varepsilon}{F_\varepsilon}'({n_\varepsilon})\chi({c_\varepsilon})\nabla{c_\varepsilon}\cdot(\nabla n-\nabla{n_\varepsilon})dxdsdt\\
&=&-\int^T_0\int^t_0<\partial_t n_\varepsilon,n>dsdt+\frac12\int^T_0\int_\Omega{n_\varepsilon}^2dxdt-\frac{T}2\int_\Omega n_{0\varepsilon}^2dx\\
&&-\int_{\Sigma_T}{u_\varepsilon}\cdot\nabla {n_\varepsilon}(n-{n_\varepsilon})dxdsdt+\int_{\Sigma_T}{n_\varepsilon}{F_\varepsilon}'({n_\varepsilon})\chi({c_\varepsilon})\nabla{c_\varepsilon}\cdot(\nabla n-\nabla{n_\varepsilon})dxdsdt\\
&=:&J_1+J_2+J_3+J_4+J_5.
\end{eqnarray*}
From (\ref{eq-cor5.0-nabla-n-p}), (\ref{eq-lem5.1-u}), (\ref{n-strong-p}) we know that  $J_4\rightarrow0$ as $\varepsilon\searrow0$ and from (\ref{nabla-n-weak-p}), (\ref{zuhetwo}) we obtain $J_5\rightarrow0$ as $\varepsilon\searrow0$ since $\frac15+\frac3{10}+\frac1p\leq1$.
Therefore, using (\ref{approximate-initialdata-n}), (\ref{eq-lem5.1-n}), (\ref{slx-space-n}), (\ref{slx-space-nt}), (\ref{n-strong-p}), (\ref{nabla-n-weak-p}) and the Lebesgue dominated convergence theorem, we obtain
\begin{eqnarray*}
&&\lim_{\varepsilon\searrow0}\int_{\Sigma_T}(|\nabla n_\varepsilon|^2+\varepsilon)^\frac{p-2}2\nabla n_{\varepsilon}\cdot(\nabla n-\nabla n_\varepsilon)dxdsdt\\
&=&\lim_{\varepsilon\searrow0}(J_1+J_2+J_3)\\
&=&-\int^T_0\int^t_0<\partial_t n,n>dsdt+\frac12\int^T_0\int_\Omega n^2dxdt-\frac T2\int_\Omega n_0^2dx=0.
\end{eqnarray*}
Hence we have
$$\lim_{\varepsilon\searrow0}\int_{\Sigma_T}(|\nabla n_\varepsilon|^2+\varepsilon)^\frac{p-2}2\nabla n_{\varepsilon}\cdot(\nabla n-\nabla n_\varepsilon)dxdsdt=0,$$
which is equivalent to $\lim_{\varepsilon\searrow0}I_1=0$. Consequently, we have shown that
$$\lim_{\varepsilon\searrow0}\int_{\Sigma_T}(|\nabla {n_\varepsilon}|^{p-2}\nabla n_{\varepsilon}-|\nabla \sigma|^{p-2}\nabla \sigma)\cdot(\nabla n-\nabla \sigma)dxdsdt\geq0,$$
which proves (\ref{eq-lem6.2}). Choosing $\sigma=n-\lambda\xi$ with $\lambda\in \mathbb{R}$ and $\xi\in L^{p}(0,T;W^{1,p}(\Omega))$ and combining the two inequalities arising from $\lambda>0$ and $\lambda<0$, we obtain the assertion of the lemma.
\end{proof}


\end{document}